\UseRawInputEncoding 
\documentclass{amsart}
\usepackage{amsmath,amsfonts,amscd,amssymb,amsthm,epsfig,euscript,eufrak,mathdots}
\usepackage{amsxtra,mathrsfs,xcolor}
\usepackage{centernot}
\usepackage[all]{xy}
\usepackage{verbatim,epsf}
\newtheorem{theorem}{Theorem}
\newtheorem{proposition}{Proposition}
\newtheorem{lemma}{Lemma}
\newtheorem{corollary}{Corollary}
\theoremstyle{definition}
\newtheorem{definition}{Definition}

\theoremstyle{remark}
\newtheorem{remark}{Remark}
\numberwithin{equation}{section}
\newcommand{\field}[1]{\ensuremath{\mathbb{#1}}}
\newcommand{\CC}{\field{C}}

\newcommand{\HH}{\field{H}}
\newcommand{\PP}{\field{P}}
\newcommand{\RR}{\field{R}}

\newcommand{\ZZ}{\field{Z}}

\DeclareMathOperator{\im}{Im}
 \DeclareMathOperator{\PSL}{PSL}

\newcommand{\del}{\partial}
\newcommand{\delb}{\bar\partial}

\newcommand{\R}{{\mathbb{R}}}
\newcommand{\curly}[1]{\mathscr{#1}}
\newcommand{\cA}{\curly{A}}

\newcommand{\cC}{\curly{C}}

\newcommand{\cE}{\curly{E}}
\newcommand{\cF}{\curly{F}}

\newcommand{\cH}{\curly{H}}

\newcommand{\cJ}{\curly{J}}
\newcommand{\cK}{\curly{K}}
\newcommand{\cL}{\curly{L}}

\newcommand{\cN}{\curly{N}}

\newcommand{\cS}{\curly{S}}

\newcommand{\cU}{\curly{U}}

\newcommand{\cW}{\curly{W}}

\newcommand{\cY}{\curly{Y}}

\newcommand{\s}{\gamma}

\newcommand{\bk}{\backslash}
\newcommand{\Ga}{\Gamma}

\newcommand{\pa}{\partial}
\newcommand{\la}{\langle}
\newcommand{\ra}{\rangle}

\newcommand{\vep}{\varepsilon}

\DeclareMathOperator{\tr}{tr} \DeclareMathOperator{\Hom}{Hom}
 
\DeclareMathOperator{\Aut}{Aut}
 
\DeclareMathOperator{\GL}{GL} \DeclareMathOperator{\End}{End}
 \DeclareMathOperator{\Ad}{Ad}

\usepackage{hyperref}
\hypersetup{colorlinks,citecolor=blue,plainpages=false,hypertexnames=false}

\begin{document}
\title[Logarithmic connections and WZNW action]{Logarithmic connections, WZNW action, and moduli of parabolic bundles on the sphere}
\author{Claudio Meneses}
\address{Mathematisches Seminar \\ 
Christian-Albrechts Universit\"at zu Kiel\\ Heinrich-Hecht- \indent Platz 6 \\ 24118 Kiel, Germany}
\curraddr{ 
}
\email{meneses@math.uni-kiel.de}
\author{Leon A. Takhtajan}
\address{Department of Mathematics \\
Stony Brook University\\ Stony Brook, NY 11794-3651 \\ \indent USA;
\newline
\indent The Euler International Mathematical Institute, Saint Petersburg, Russia}
\email{leontak@math.stonybrook.edu}
\maketitle
\begin{abstract}  Moduli spaces of stable parabolic bundles of parabolic degree $0$ over the Riemann sphere are stratified according to the Harder--Narasimhan filtration of underlying vector bundles. Over a Zariski open subset $\cN_{0}$ of the open stratum depending explicitly on a choice of parabolic weights, a real-valued function $\cS$ is defined as the regularized critical value of the non-compact Wess--Zumino--Novikov--Witten action functional. The definition of $\cS$ depends on a suitable notion of parabolic bundle `uniformization map' following from the Mehta--Seshadri and Birkhoff--Grothendieck theorems. It is shown that $-\cS$ is a primitive for a (1,0)-form $\vartheta$ on $\cN_{0}$ associated with the uniformization data of each intrinsic irreducible unitary logarithmic connection. Moreover, it is proved that $-\cS$ is a K\"ahler potential for $(\Omega-\Omega_{\mathrm{T}})|_{\cN_{0}}$,  where $\Omega$ is the Narasimhan--Atiyah--Bott K\"ahler form in $\cN$ and $\Omega_{\mathrm{T}}$ is a certain linear combination of tautological $(1,1)$-forms associated with the marked points. These results provide an explicit relation between the cohomology class $[\Omega]$ and tautological classes, which holds globally over certain open chambers of parabolic weights where $\cN_{0} = \cN$.
\end{abstract}

\tableofcontents
\section{Introduction}

The analytic geometry of moduli spaces of Riemann surfaces and vector bundles is closely tied with the two-dimensional conformal field theory, formulated in the 80s by Belavin, Polyakov and Zamolodchikov \cite{BPZ84}. One of the fundamental models of the theory is the quantum Liouville model, a quantization of the classical theory defined by the Liouville action on a Riemann surface, whose Euler-Lagrange equation is the Liouville equation that determinines the hyperbolic metric on it. Semi-classical analysis of the quantum Liouville theory indicates a deep and unexpected relation between the critical value of Liouville action and the accessory parameters of the Fuchsian uniformization of Riemann surfaces. A precise form of this relation, as well as an unexpected connection with the Weil--Petersson metric on Teichm\"uller space, was proved by P. Zograf and the second author in \cite{ZT87a,ZT87b}. We refer to \cite{T96,TT06} and references therein for further results and details.

Finding an analog of such results for moduli spaces of stable parabolic bundles on Riemann surfaces, in the spirit of \cite{ZT89,TZ07}, remained as an interesting open problem. It is well known that such moduli spaces appear in conformal field theories associated with the Wess--Zumino--Novikov--Witten (WZNW) action for compact Lie groups, introduced by Novikov \cite{N82} and Witten \cite{W84}. Starting from the $\mathrm{SU}(2)$ case  \cite{KZ84}, the compact WZNW models  have been thoroughly investigated (see, e.g., the monograph \cite{DF}). However, as far as the analogy in question is concerned,  it is not the compact WZNW models, but rather their non-compact analogs \cite{Gaw92}, that are the appropriate candidates to consider. Non-compact WZNW models do not lead to rational conformal field theories, and are in general less understood. 
 
In the case of genus 0 this analogy can be made precise as follows. Let $E_{*}$  be a rank $r$ stable parabolic bundle of parabolic degree 0 over $\CC\PP^{1}$ with a fixed set of marked points $z_{1},\dots, z_{n}\in\CC\PP^{1}$. The Mehta--Seshadri theorem establishes the equivalence between the notion of parabolic stability and the existence of a singular Hermitian metric on $E_{*}$ whose associated Chern connection is flat and irreducible over $\CC\PP^{1}\setminus\{z_{1},\dots z_{n}\}$, and has logarithmic singularities at the marked points with the residues compatible with the parabolic structure. Over the Riemann sphere, the Birkhoff--Grothendieck decomposition of holomorphic vector bundles provides an explicit trivialization on the underlying vector bundle $E$, and allows to interpret the Mehta--Seshadri theorem as the existence of a `parabolic bundle uniformization map' $\cJ$.  It follows that the singular Hermitian metric on $E_{*}$ can be described as a smooth map $h:\CC\PP^{1}\setminus\{z_{1},\dots,z_{n}\}\rightarrow \CMcal{H}_{r}$, where $\CMcal{H}_{r}$ is the space of Hermitian positive-definite $r\times r$ matrices, having prescribed asymptotic behavior at the marked points and satisfying the equation
\[
\bar{\partial}\left(h^{-1}\partial h\right)=0.
\]
This equation is precisely the Euler--Lagrange equation of the celebrated WZNW action functional for  $\CMcal{H}_{r}$ -valued maps. Such a map $h$ would be defined only up to the action of the group  $\Aut(E)$ of bundle automorphisms, since the latter is always non-trivial.

However, in contrast to the moduli problem for Riemann surfaces, the nature of the moduli problem in question leads to additional geometric features. In general, the dependence on a choice of parabolic weights induces wall-crossing phenomena. Moreover, the peculiarities of genus 0 define special moduli space stratifications with an explicit dependence on the combinatorial structure of parabolic weight polytopes. Such stratifications, as well as their dependence on parabolic weights, play a decisive role in the main results of this work.

More precisely, over a moduli space $\cN$ of stable parabolic bundles, we are lead to a certain Zariski open subset $\cN_{0}$ of geometric significance.  Namely, there is a stratification of $\cN$ determined by the Harder--Narasimhan filtration associated with the Birkhoff--Grothendieck splitting type of a holomorphic vector bundle on $\CC\PP^{1}$, which depends on a choice of parabolic weights. Over its Zariski open stratum with a fixed Birkhoff--Grothendieck splitting type $E_{N_{0}}$ there is an open subset $\cN_{0}$ such that over it a consistent choice of representatives of $\Aut\left(E_{N_{0}}\right)$-orbits for $\cJ$ and  $h$ can be made. We refer to $\cN_{0}$ as the \emph{regular locus}. As the moduli space $\cN$, the regular locus depends rather nontrivially on the choice of parabolic weights. 
%\footnote{This dependence is currently under investigation by the first author.}
For the purposes of this paper we note that in many cases there exist open chambers in the weight polytopes where $\cN_{0} = \cN$. On the regular locus $\cN_{0}$, the explicit choice of the maps $h:\CC\PP^{1}\setminus\{z_{1},\dots,z_{n}\}\rightarrow \CMcal{H}_{r}$ allows us to define a smooth real-valued function $\cS:\cN_{0}\rightarrow \RR$ as the critical values of the WZNW action, and a smooth $(1,0)$-form $\vartheta$ on $\cN_{0}$, associated with the logarithmic connection $d+h^{-1}\partial h$. Our first main result, Theorem \ref{main-theo1}, is the following explicit relation on $\cN_{0}$,
$$\partial \cS = -\vartheta.$$

The moduli space $\cN$ carries the Narasimhan--Atiyah--Bott K\"ahler form $\Omega$ and the $(1,1)$-forms $\Omega_{ij}$, which are the first Chern forms of tautological line bundles associated with the marked points. 
Our second main result, Theorem \ref{main-theo2}, establishes a relation between these natural objects and the  $(1,0)$-form $\vartheta$ on $\cN_{0}$. Namely,
\[
\bar{\partial}\vartheta = 2\sqrt{-1}\left(\Omega - \Omega_{\mathrm{T}}\right)|_{\cN_{0}},\qquad\Omega_{\mathrm{T}}= \sum \beta_{ij}\Omega_{ij},
\]
where $\beta_{ij}$ depend linearly on the parabolic weights and the bundle splitting coefficients. Together, Theorems \ref{main-theo1} and \ref{main-theo2} imply that $-\cS$ is a K\"ahler potential over $\cN_{0}$ for the difference between the (1,1)-forms $\Omega$ and $\Omega_{\mathrm{T}}$ (Corollary \ref{cor:potential}). It expresses the cohomology class $[\Omega]$ on $\cN_{0}$ as a concrete linear combination of tautological classes $[\Omega_{ij}]$. This result establishes a new relation between non-compact WZNW models and the analytic geometry
of moduli spaces.

The paper is organized as follows. In Section \ref{pb-conn} we review the Mehta--Seshadri theorem for stable parabolic bundle $E_{*}$ on $\CC\PP^{1}$,
introduce bundle uniformization maps and related geometric structures --- the singular Hermitian metric and unitary logarithmic connection on $E_{*}$.
 In Section \ref{Mod} we remind the complex analytic theory of the moduli space $\cN$ of stable parabolic bundles, and define the regular locus $\cN_{0}$ and the $(1,0)$-form $\vartheta$. 
 In Section \ref{WZ} we give a construction of the regularized WZNW action, and in Section \ref{Main} we prove Theorems \ref{main-theo1} and \ref{main-theo2}.

\subsection*{Acknowledgments}  The first author (C.M.) was partially supported by the DFG SPP 2026 priority programme ``Geometry at infinity". The work of the second author (L.T.) was done under the partial support of the NSF grant DMS-1005769.

\section{Parabolic bundles and logarithmic connections} \label{pb-conn}

\subsection{Parabolic bundles} \label{pb}

A parabolic bundle $E_{*}$ of rank $r$ on $\CC\PP^{1}$ and a fixed set of marked points $z_{1},\dots, z_{n}\in\CC\PP^{1}$ is a holomorphic vector bundle $E$ together with a \emph{parabolic structure} --- complete descending flags\footnote{In general, one considers arbitrary flags and weights with multiplicities.} $E_{z_{i}}=F_{i1}\supset F_{i2}\supset\cdots\supset F_{ir}\supset\{0\}$ in the fibers $E_{z_{i}}$, $i = 1,\dots, n$, with weights $0 \leq \alpha_{i1}<\alpha_{i2}<\cdots<\alpha_{ir}<1$.  
The parabolic degree of a parabolic bundle $E_{*}$ is defined as
$$\mathrm{par}\,\mathrm{deg}\,E_{*}=d +\sum_{i=1}^{n}\sum_{j=1}^{r}\alpha_{ij},$$
where $d=\mathrm{deg}\, E$ is the degree of the vector bundle $E$. A morphism $f:E_{*}\rightarrow E'_{*}$ of parabolic vector bundles is a morphism of holomorphic vector bundles
such that for every $z_{i}$, $f(F_{ij})\subset F'_{ik+1}$ whenever $\alpha_{ij}>\alpha'_{ik}$.
A parabolic subbundle $F_{*}$ of $E_{*}$ is a subbundle $F\subset E$ such that for every $z_{i}$ the parabolic structure in $F_{*}$  is induced from the parabolic structure in $E_{*}$ by restriction.

A parabolic bundle $E_{*}$ of parabolic degree $0$ is \emph{stable} (resp. \emph{semi-stable}) 
if every proper parabolic subbundle $F_{*}$ of $E_{*}$ satisfies $\mathrm{par}\,\mathrm{deg}\,F_{*}<0$. (resp. $\leq 0$). 
When $E_{*}$ is stable, its group $\mathrm{Par}\Aut E_{*}$ of parabolic automorphisms consists of nonzero multiples of the identity.  A theorem of Mehta--Seshadri \cite{MS80} generalizes the celebrated theorem of Narasimhan--Seshadri \cite{NS65} for stable vector bundles on a compact Riemann surface to the case of parabolic bundles. It states that when $2g-2 + n > 0$, stable parabolic bundles over a compact Riemann surface $X$ of genus $g$ are precisely those associated with irreducible unitary representations of the fundamental group of the non-compact Riemann surface $X_{0} = X\setminus\{z_{1},\dots,z_{n}\}$. 

The precise formulation in the special case $X=\CC\PP^{1}$ is the following.  Let 
$$\HH=\{\tau\in\CC : \im\tau>0\}$$ 
be the Poincar\'{e} half-plane model of the Lobatchevsky plane and let $X_{0}=\CC\PP^{1}\setminus\{z_{1},\dots,z_{n}\}$, where the normalization $z_{n-2}=0, z_{n-1}=1$ and $z_{n}=\infty$ is always assumed. By the uniformization theorem, 
$$X_{0}\cong\Ga\bk\HH,$$ 
where $\Ga$ is a torsion-free Fuchsian group generated by $n$ parabolic transformations $S_{1},\dots,S_{n}$ satisfying the single relation
$$S_{1}\cdots S_{n}=1$$
and having the property that their fixed points $\tau_{1},\dots,\tau_{n}\in\RR\cup\{\infty\}$ are mapped to $z_{1},\dots,z_{n}\in\CC\PP^{1}$ and $\tau_{n-2}=0, \tau_{n-1}=1$ and $\tau_{n}=\infty$.
The uniformization map --- a classical \emph{Klein's Hauptmodul} (or \emph{Hauptfunktion}) --- is a complex-analytic covering $J:\HH\rightarrow X_{0}$ which is $\Ga$-automorphic
and takes every value in $\CC\setminus\{z_{1},\dots,z_{n-3},0,1\}$ exactly once in the fundamental domain of the group $\Ga$. It extends to the holomorphic map $J: \HH^{\ast}\rightarrow\CC\PP^{1}$, where $\HH^{\ast}$ is the  union of $\HH$ with all cusps for $\Ga$. 

Given a set of parabolic weights $\{\alpha_{ij}\}$, let $W_{i}=\mathrm{diag}(\alpha_{i1},\dots,\alpha_{ir})$ and $D_{i}=e^{2\pi\sqrt{-1}\,W_{i}}$ for each $i=1,\dots,n$. A unitary representation $\rho:\Gamma\rightarrow \mathrm{U}(r)$ is called \emph{admissible} if for each $i=1,\dots,n$ we have $\rho(S_{i})=U_{i}D_{i}U_{i}^{-1}$ with  $U_{i}\in\mathrm{SU}(r)$. Clearly each $U_{i}$ is defined only up to right multiplication by a diagonal matrix.
Thus an admissible unitary representation $\rho$ defines a set of points $[U_{1}],\dots,[U_{n}]$ in the homogeneous spaces of conjugacy classes of $D_{1},\dots,D_{n}$ in $\mathrm{U}(r)$, which are isomorphic to $\mathrm{SU}(r)/U(1)^{r-1}$. The group $\Gamma$ acts on the trivial bundle $\HH\times\CC^{r}$ over $\HH$ by $(\tau,v)\mapsto(\gamma\tau, \rho(\gamma)v)$, defining a local system $E^{\rho}_{0}=\Ga\bk(\HH\times\CC^{r})\rightarrow\Ga\bk\HH\cong X_{0}$. $E^{\rho}_{0}$ extends to a holomorphic vector bundle $E^{\rho}$ over $\CC\PP^{1}$ together with a collection of flags at the fibers over the marked points induced by the data $[U_{i}]$, and determines a semi-stable parabolic bundle $E^{\rho}_{*}$ which is stable when $\rho$ is irreducible, in such a way that $E^{\rho_{1}}_{*}\cong E^{\rho_{2}}_{*}$ if and only if $\rho_{1}\cong \rho_{2}$ (see \cite{MS80} for details).

The Mehta-Seshadri theorem asserts that the converse is also true, namely, that for every stable parabolic bundle $E_{*}$ of parabolic degree 0 there is an irreducible admissible representation $\rho$ such that $E_{*}\cong E^{\rho}_{*}$.

By the Birkhoff--Grothendieck theorem, every holomorphic vector bundle $E$ of rank $r$ over $\CC\PP^{1}$ is isomorphic to a direct sum of line bundles, 
$$E\cong\bigoplus_{j=1}^{r}\mathcal{O}(m_{j}),\qquad m_{1}\leq m_{2}\leq\dots\leq m_{r}.$$ Such an isomorphism depends on a choice of point in $\CC\PP^{1}$, which we assume to be $\infty$. Let $N=\mathrm{diag}(m_{1},\dots,m_{r})$. Upon the choice of a second point, e.g. $0\in\CC\PP^{1}$, the  bundle $E$ is determined 
by the transition function\footnote{In what follows $z^{N}$, $q^{N}$, etc., will always stand for the corresponding diagonal matrices.}
$$g(z)=z^{N}=\mathrm{diag}(z^{m_{1}},\dots,z^{m_{r}}),$$
defined on the intersection $\CC^{*}$ of the charts $\CC=\CC\PP^{1}\setminus\{\infty\}$ and $\CC^{\ast}\cup\{\infty\}=\CC\PP^{1}\setminus\{0\}$ of  $\CC\PP^{1}$. We will denote such a bundle-splitting canonical form by $E_{N}$. It follows that every parabolic bundle $E_{*}\rightarrow\CC\PP^{1}$ is isomorphic to a parabolic bundle of the form $\left(E_{N}\right)_{*}$.

The endomorphisms of $E$ are global sections of the bundle $\mathrm{End}\,E =E^{\vee}\otimes E$, where $E^{\vee}$ is the dual to $E$. When $X=\CC\PP^{1}$, the Riemann-Roch theorem for vector bundles states
$$\dim \check{H}^0(\CC\PP^1,\mathrm{End}\,E)-\dim \check{H}^1(\CC\PP^1,\mathrm{End}\,E)=r^{2}.$$
It follows that $\dim \check{H}^0(\CC\PP^1,\mathrm{End}\,E)$ attains its minimal value $r^{2}$ if and only if $\dim \check{H}^{1}(\CC\PP^1,\mathrm{End}\,E)= \dim \check{H}^{0}(\CC\PP^1,\mathcal{O}(-2)\otimes\mathrm{End}\,E)=0$ or equivalently, if and only if $|m_{j}-m_{k}|\leq 1$ for all $j,k=1,\dots,r$. Such bundles are called 
\emph{evenly-split} \cite{Bel01, Bis02}. For every  $d\in\ZZ$ and $r>1$ there is a a unique evenly-split bundle $E_{N_{0}}$ of degree $d$ and rank $r$ up to isomorphism:
\begin{equation}\label{split-dec}
E_{N_{0}}=\CMcal{O}(m)^{r-p}\oplus\CMcal{O}(m+1)^{p},
\end{equation}
where $d=mr+p,\;0\leq p <r$, and
\[
N_{0}=\mathrm{diag}(\underbrace{m,\dots,m}_{r-p},\underbrace{m+1,\dots,m+1}_{p}).
\]
Let $\Aut E$ denote the group of holomorphic bundle automorphisms of a vector bundle $E$. When $r\mid d$ we have that $\Aut E_{N_{0}}\cong \GL(r,\CC)$. Otherwise, in terms of the affine trivialization over $\CC$, $\Aut E_{N_{0}}$ gets identified with a group of matrix-valued polynomials of block-lower triangular type. Namely, if $r\centernot\mid d$, it follows that
\begin{equation}\label{fac-aut}
\Aut E_{N_{0}}=\mathrm{P}_{N_{0}}\ltimes \mathrm{N}_{N_{0}},
\end{equation}
where $\mathrm{P}_{N_{0}}$ is the subgroup of block-lower triangular matrices relative to the partition $(r-p,p)$, and $\mathrm{N}_{N_{0}}$ is the normal subgroup of functions having the following $(r-p,p)$ block form $$g(z)=\begin{pmatrix} I_{r-p} & 0\\ zC & I_{p}\end{pmatrix},\quad\text{where $C$ is an arbitrary $p\times (r-p)$ matrix}.$$
It follows from this characterization that when $p\neq 0$, the group $\Aut E_{N_{0}}$ preserves the second summand in the decomposition \eqref{split-dec}. Whence in such case every evenly-split bundle $E$ has a subbundle  $\mathcal{O}(m+1)^{p}\hookrightarrow E$, which is independent of the isomorphism $E\cong E_{N_{0}}$. Moreover, the corresponding Harder--Narasimhan filtration of the bundle $E$ reduces to $$E\supset \mathcal{O}(m+1)^{p}\supset\{0\}.$$
For any $z\in\CC\PP^{1}$, let $V_{z} = \mathcal{O}(m+1)^{p}|_{z}\subset E_{z}$ be the fiber at $z$, $\mathrm{P}(V_{z})\subset \mathrm{GL}(E_{z})$ be its parabolic subgroup, and $\mathrm{N}(V_{z})$ be its unipotent radical. Invariantly, restriction to the fibers $E_{0}$, $E_{\infty}$ determines the isomorphisms $\mathrm{P}_{N_{0}}\cong \mathrm{P}(V_{0})$, $\mathrm{N}_{N_{0}}\cong \mathrm{N}(V_{\infty})$. 

The vector bundle $E^{\rho}$ can also be described in terms of another set of transition functions (cf. \cite[Remark 6.2]{NS65}). Namely, since $X_{0}$ is non-compact, by a theorem of Stein the holomorphic vector bundle $E^{\rho}_{0}$ is trivial. Hence there is a holomorphic function $G:\HH\to\mathrm{GL}(r,\CC)$ satisfying 
$$G(\s\tau)=G(\tau)\rho(\s)^{-1},\quad \forall\s \in\Ga,$$
so that the bundle map $G\circ J^{-1}: E^{\rho}_{0}\to X_{0}\times\CC^{r}$ is an isomorphism.
 For any choice of representatives $U_{1},\cdots,U_{n}$, the function $G$ can be written in the neighborhood of each $\tau_{i}$ as 
\begin{equation} \label{factor}
G(\sigma_{i}\tau)=G_{i}(q)\,q^{-W_{i}}U_{i}^{-1}.
\end{equation}
Here $G_{i}(q)$ are holomorphic and invertible in some punctured neighborhood of $q=0$, $q^{-W_{i}}=\mathrm{diag}(e^{-2\pi\sqrt{-1}\alpha_{i1}\tau},\dots,e^{-2\pi\sqrt{-1}\alpha_{ir}\tau})$  and $\sigma_{i}\in\PSL(2,\RR)$ are such that $\sigma_{i}(\infty)=\tau_{i}$ and $\sigma^{-1}_{i}S_{i}\sigma_{i}=\left(\begin{smallmatrix}1 & \pm 1\\ 0 & \;\;1\end{smallmatrix}\right)$, $i=1,\dots,n$. 
Let $\mathcal{U}=\{\cU_{0},\cU_{1},\dots,\cU_{n}\}$ be an open cover of $\CC\mathbb{P}^{1}$ where $\cU_{0}=X_{0}$, and for $i \geq 1$, $\cU_{i}$ are sufficiently small open disks  around $z_{i}$ so that $\cU_{ij}=\emptyset$ for $i,j\neq 0$, $i\neq j$. By the definition of the bundle $E^{\rho}$ (see \cite{MS80}), local trivializations of $E^{\rho}$ over $\cU_{i}$ are given by the maps $\psi_{i}\circ J^{-1}$, where $\psi_{i}(\sigma_{i}\tau)=q^{-W_{i}}U_{i}^{-1}$. Whence the transition functions $g_{0i}:\cU_{0i}\rightarrow \mathrm{GL}(r,\CC)$ of the bundle $E^{\rho}$ for the cover $\mathcal{U}$ are given by the formula
$$g_{0i}=G_{i}\circ\sigma_{i}^{-1}\circ J^{-1},\quad i=1,\dots,n.$$

The definition of $E^{\rho}$ and the maximum principle imply that $\check{H}^{0}(\CC\PP^{1},E^{\rho})\cong (\CC^{r})^{\rho}$, 
where the right-hand side denotes the subspace of $\rho$-invariant vectors in $\CC^{r}$. Hence $m_{j}<0$ for all $j=1,\dots,r$ when $\rho$ is irreducible. Moreover, if $0 < \alpha_{i1}$ for $i = 1,\dots, n$, we have that
$$E^{\vee}\cong \tilde{E}\otimes\mathcal{O}(n),$$
where $E=E^{\rho}$, $\tilde{E}=E^{\bar\rho}$ and $\bar\rho={}^{t}\rho^{-1}$ is the contragradient representation to $\rho$.
This follows from the comparison of the transition functions  ${}^{t}g_{0i}^{-1}$ for $E^{\vee}$ with those of $\tilde{E}$ which correspond to the function ${}^{t}G^{-1}$.  Hence if $E_{*}$ is a stable parabolic bundle of parabolic degree 0 whose parabolic weights satisfy $0 < \alpha_{i1}$ for each $i = 1,\dots,n$, then necessarily
$$-n<m_{j}<0,\qquad j=1,\dots,r.$$

\subsection{Parabolic bundle uniformization map} \label{matrix-haup} 

The parabolic structures on a given bundle splitting $E_{N}$ that arise from an admissible representation can be described by means of a uniformization map of parabolic bundles. When any such parabolic bundle is stable, the uniformization map provides a complex-analytic interpretation of the Mehta--Seshadri theorem, and can be thought of as a matrix analog of the classical Klein's Hauptmodul $J$.
\begin{lemma}\label{Theo-Upsilon}
Let $\rho:\Ga\to\mathrm{U}(r)$ be an admissible representation such that $E^{\rho}\cong E_{N}$. Given a choice of representatives $U_{1},\dots, U_{n}\in\mathrm{SU}(r)$ there is a holomorphic function $Y:\HH\rightarrow\mathrm{GL}(r,\CC)$ satisfying 
\begin{equation} \label{auto-Y}
Y(\s\tau)=Y(\tau)\rho(\s)^{-1},\quad \forall\,\s \in\Ga,\;\tau\in\HH,
\end{equation}
and having the Fourier series expansions 
\begin{equation} \label{Y-Fourier}
Y(\sigma_{i}\tau)=\left(\sum_{k=0}^{\infty}C_{i}(k)q^{k}\right)q^{-W_{i}}U_{i}^{-1},\quad i=1,\dots,n-1,
\end{equation}
and
\begin{equation} \label{Y-Fourier-2}
Y(\sigma_{n}\tau)=q^{-N}\left(\sum_{k=0}^{\infty}C_{n}(k)q^{k}\right)q^{-W_{n}}U_{n}^{-1},
\end{equation}
where $C_{i}(0)\in\mathrm{GL}(r,\CC)$ for $i=1,\dots,n$. 
The set $\Upsilon(\rho)$ of all functions $Y$ with these properties is in one-to-one correspondence with the set of all isomorphisms $E^{\rho}\cong E_{N}$, and is a principal homogeneous space for the automorphism group $\mathrm{Aut}\,E_{N}$ of the bundle splitting $E_{N}$.
\end{lemma}
\begin{proof} Consider the function $G$, the open cover $\mathcal{U}$, and the transition functions $g_{0i}$ defined before. The existence of the function $Y$ is a consequence of the equivalence of bundles defined by the transition functions $g_{0i}$ and the Birkhoff--Grothendieck transition function $g(z)=z^{N}$. 
It follows that there exist holomorphic functions $g_{0}:\cU_{0}\to\mathrm{GL}(r,\CC)$ and $g_{i}:\cU_{i}\to\mathrm{GL}(r,\CC)$, $i=1,\dots,n$, such that  
$$g_{0i}=g_{0}\,g_{i}^{-1},\quad i\neq n\quad\text{and}\quad g_{0n}=g_{0}\,z^{N}g_{n}^{-1}.$$
Put $Y=\left(g_{0}\circ J\right)^{-1}G$. It follows from \eqref{factor} that $Y(\tau)$ has Fourier series expansions \eqref{Y-Fourier} and \eqref{Y-Fourier-2}. Conversely, a choice of map $Y:\HH\rightarrow\mathrm{GL}(r,\CC)$ satisfying \eqref{auto-Y}-\eqref{Y-Fourier-2} determines an isomorphism $E^{\rho}\cong E_{N}$.

If $U'_i=U_iV_i$ with diagonal $V_i\in\mathrm{SU}(r)$, the Fourier series expansions \eqref{Y-Fourier} and \eqref{Y-Fourier-2} would have 
Fourier coefficients $C'_i(k)=C_i(k)V_i^{-1}$, $i=1,\dots,n$.
Thus, the set $\Upsilon(\rho)$ only depends on $\rho$, and is in bijective correspondence with the set of isomorphisms $E^{\rho}\cong E_{N}$ by definition. The set of isomorphisms $E^{\rho}\cong E_{N}$ is a principal homogeneous space for $\mathrm{Aut}\,E_{N}$, therefore $Y_{1}$ and $Y_{2}$ are two functions satisfying \eqref{auto-Y}--\eqref{Y-Fourier-2} if and only if $Y_{1}\cdot Y_{2}^{-1}=g\circ J$, where $g$ is the local form over $\CC$ of an automorphism of $E_{N}$.
\end{proof}

\begin{remark} \label{Klein}
Every element $Y\in\Upsilon(\rho)$ may be considered as a matrix analog of the Klein's Hauptmodul. It follows from \eqref{auto-Y} that the map 
$$E_{0}^{\rho}\ni [\tau,v]\mapsto\cJ(\tau,v)=(J(\tau),Y(\tau)v)\in E_{N}|_{X_{0}}$$ 
establishes the isomorphism between the local system $E_{0}^{\rho}$ over $\Gamma\bk\HH$ and the restriction $E_{N}|_{X_{0}}$. Properties \eqref{Y-Fourier}--\eqref{Y-Fourier-2} ensure that $\cJ$ extends to an isomorphism of parabolic bundles $E_{\ast}^{\rho}\cong \left(E_{N}\right)_{\ast}$ which plays the role of a \emph{bundle uniformization map}. 
In particular, the ordered frames defined by the matrices $C_{1}(0),\dots, C_{n}(0)$ determine the corresponding flags on the fibers $(E_{N})_{z_{1}},\dots,(E_{N})_{z_{n}}$. 
\end{remark}

\subsection{Singular Hermitian metrics and unitary logarithmic connections}\label{can-con} 

Denote by $D$ the divisor $z_{1}+\dots+z_{n}$ in $\CC\PP^{1}$. A \emph{logarithmic connection} on a holomorphic bundle $E\rightarrow \CC\PP^{1}$ is a map of sheaves 
$$\nabla : \mathcal{O}(E)\rightarrow \mathcal{O}\left(E\otimes K_{\CC\PP^{1}}(D)\right)$$
that is $\CC$-linear and satisfies the Leibniz rule with respect to the $\mathcal{O}_{\CC\PP^{1}}$-module structure on $\mathcal{O}(E)$. With every logarithmic connection $\nabla$ there is an associated set of \emph{residues} $\{\mathrm{Res}_{z_{i}}\nabla \in \End E_{z_{i}}\,:\, i = 1,\dots,n\}$ (see \cite{Deligne,Simpson90,BL11}).  A choice of the base point $z_{0}\in\CC\PP^{1}$ and local holomorphic frames near each $z_{1},\dots,z_{n}$ determines a \emph{monodromy representation} for a logarithmic connection $\nabla$. Its conjugacy class is an invariant of $\nabla$. 

A logarithmic connection $\nabla$ on the underlying bundle $E$ of a parabolic bundle $E_{*}$ is said to be \emph{adapted to the parabolic structure of $E_{*}$} if for every $i=1,\dots,n$, $\mathrm{Res}_{z_{i}}\nabla$ is semisimple with eigenvalues $0 <\alpha_{i1}<\dots<\alpha_{ir}<1$ and eigenlines $L_{i1},\dots,L_{ir}$, such that the corresponding flag subspaces are $F_{ij} = L_{ij}\oplus\dots\oplus L_{ir}$ for each $j =1,\dots,r$.  

The space $\cC(E_{\ast})$ of logarithmic connections adapted to $E_{*}$ is nonempty if $E_{*}$ is an indecomposable parabolic vector bundle of parabolic degree zero (see \cite[Proposition 4.1]{BL11}), and is an affine space modeled on the vector space of (strongly) \emph{parabolic Higgs fields} on $E_{*}$ --- a subspace
$$\check{H}^{0}\left(\CC\PP^{1}, \left(\mathrm{Par}\End E_{*}\right)^{\vee}\otimes K_{\CC\PP^{1}}\right)\subset \check{H}^{0}\left(\CC\PP^{1}, \End E\otimes K_{\CC\PP^{1}}(D)\right),$$
consisting of 
$\End E$-valued meromorphic $(1,0)$-forms $\Phi$ on $\CC\PP^{1}$ with at most simple poles on $D$, whose residues $\mathrm{Res}_{z_{i}}\Phi$ are nilpotent and preserve the flags on $E_{z_{i}}$ for all marked points.\footnote{It can be verified that for any complete weighted flag $F$ on an $r$-dimensional vector space $V$, the set of semisimple endomorphisms of $V$ preserving $F$ with fixed eigenvalues $0<\alpha_{1}<\dots <\alpha_{r}<1$ is an affine space for  the unipotent radical $\mathfrak{n}(F)\subset \mathfrak{p}(F)$ of the parabolic Lie algebra of $F$, and the latter is the space of nilpotent endomorphisms preserving $F$.}  

It follows from the Mehta-Seshadri theorem that every stable parabolic bundle $E_{*}$ admits a logarithmic connection adapted to $E_{*}$ with monodromy given by an irreducible admissible representation $\rho$. Under the isomorphism $E_{*}\cong E^{\rho}_{*}$, the standard Hermitian inner product in $\CC^{r}$ and the trivial connection $d$ on $\HH\times\CC^{r}$ define a Hermitian metric $h_{0}$ in the local system $E_{0}\cong E_{0}^{\rho}=\Ga\bk(\HH\times\CC^{r})$ with corresponding flat Chern connection $\nabla_{0} = d+A_{0}$, so that $A_{0}=h^{-1}_{0}\del h_{0}$ in terms of a holomorphic frame on $X_{0}$. 
These structures extend to a singular Hermitian metric\footnote{A gauge-theoretic approach to the Mehta-Seshadri theorem is presented in \cite{Biq91}, where such singular Hermitian metrics are called \emph{adapted to a parabolic structure}.} $h_{E}$ and a logarithmic connection $\nabla_{E} = d+A_{E}$ in the bundle $E^{\rho}$. In terms of the local trivialization maps $\psi_{i}$ over the neighborhoods $\cU_{i}$ of $z_{i}$ considered before, $h_{E}$ and $A_{E}$ are given by $$h_{i}=|\zeta_{i}|^{2W_{i}}\quad\text{and}\quad A_{i}=h_{i}^{-1}\partial h_{i}=\frac{W_{i}}{\zeta_{i}}d\zeta_{i},\qquad \text{where}\quad \zeta_{i}=q\circ\sigma_{i}^{-1}.$$  

Both $h_{E}$ and $\nabla_{E}$ can also be described in terms of the splitting $E_{N}$. Namely, the isomorphism $\cJ$ in Remark \ref{Klein} gives  a trivialization of $E^{\rho}$ over $\cU_{0}$, which extends to $\CC$. Put $\cY(z)=Y(J^{-1}(z))$. It follows from the Fourier series expansion of $J(\tau)$ (see, e.g., \cite[Lemma 2]{ZT87a}) and \eqref{auto-Y}--\eqref{Y-Fourier} that $\cY(z)$ is a `multi-valued' meromorphic function on $\CC\PP^{1}$ with the following behavior near the points $z_{i}$,
\begin{align}
\cY(z) &=\left(C_{i}(0)+\sum_{k=1}^{\infty}\tilde{C}_{i}(k)(z-z_{i})^{k}\right)e^{-W_{i}\log(z-z_{i})}U^{-1}_{i},\quad i\neq n,\label{Y-1}\\ 
\intertext{and near $z_{n}=\infty$,}
\cY(z)&=z^{N}\left(C_{n}(0)+\sum_{k=1}^{\infty}\tilde{C}_{n}(k)z^{-k}\right)e^{W_{n}\log z}U_{n}^{-1}.\label{Y-2}
\end{align}
The corresponding Hermitian metric $h_{E}$ in this trivialization is given by the matrix-valued function 
\begin{align} 
h(z,\bar{z})& =(\cY(z)\cY(z)^*)^{-1}, \label{h} \\
\intertext{where $\cY^{*}=^{t}\!\!\bar{\cY}$ is the Hermitian conjugate of $\cY$, and the logarithmic connection $\nabla_{E}$ by the matrix-valued $(1,0)$-form}
A(z)dz&=\cY(z)(\cY(z)^{-1})'dz=-\cY'(z)\cY(z)^{-1}dz. \label{A}
\end{align}
It follows from equation \eqref{Y-1} that the matrix-valued function $A=-\cY'\cY^{-1}$ is holomorphic on $X_{0}$ with simple poles at the points $z_{i}$, $i\neq n$: 
\[
A(z)=\frac{A_{i}}{z-z_{i}} +O(1), \quad\text{where}\quad A_{i}=C_{i}(0)W_{i}C_{i}(0)^{-1},
\]
Moreover, it follows from \eqref{Y-2} that as $z\to \infty$,
\[
z^{-N}A(z)z^{N} + \frac{N}{z}= - \frac{A_{n}}{z} + O\left(\frac{1}{z^{2}}\right),\quad\text{where}\quad A_{n}=C_{n}(0) W_{n} C_{n}(0)^{-1}.
\]

\section{The complex analytic theory of the moduli space}\label{Mod}

In what follows, $\cN$ will denote a moduli space of rank $r$ stable parabolic bundles of parabolic degree 0 over $\CC\PP^{1}$, depending on a choice of admissible parabolic weights $\cW =\{W_{1},\dots,W_{n}\}$. Necessary and sufficient conditions (in the form of parabolic weight inequalities) for a choice of admissible parabolic weights to determine a non-empty moduli space are described in \cite{Bel01, Bis02}, and we will assume that they are satisfied henceforth. According to the Mehta-Seshadri theorem, $\cN$ is real-analytically isomorphic to the $\mathrm{U}(r)$-character variety 
$$\cN\cong\cK=\Hom(\Ga,\mathrm{U}(r))^{0}/\mathrm{U}(r),$$
of equivalence classes of admissible irreducible unitary representations, and for generic parabolic weights is a complex manifold of dimension   
\begin{equation}\label{dim}
\dim_{\CC}\cN=\tfrac{1}{2}n(r^2-1)-r^2+1.
\end{equation}
For every choice of degree $-nr < d < 0$, the parabolic weight inequalities, granting the existence of a non-empty moduli space $\cN$, determine a polytope containing a finite collection of semi-stability walls, whose complement is a finite set of open chambers. For any choice of parabolic weights in an open chamber, every semistable parabolic bundle of parabolic degree 0 is strictly stable, and the induced moduli space $\cN$ is a compact complex manifold. The biholomorphic type of $\cN$ is an invariant of the open chamber \cite{BH95}.

\subsection{The complex structure} \label{cs}

The special Hermitian metric $h_{E_{0}}$ in the local system $E_{0}$ (Section \ref{can-con}) determines a Hermitian metric $h_{\End E_{0}}$ in the induced local system $\End E_{0}\cong E^{\Ad\rho}_{0}$, where $\Ad\rho := \Ad\circ\rho$ is the induced adjoint representation in $\End \CC^{r}$, which together with the hyperbolic metric on $X_{0}$ defines the Hodge $\ast$-operator on the $(p,q)$-forms on $X_{0}$ with values in $\End E_{0}$. Denote by $\cH^{p,q}(X,\End E_{0})$ the corresponding spaces of square integrable harmonic $(p,q)$-forms on $X_{0}$ with values in $\End E_{0}$.\footnote{We use the same normalization for the inner product on $(p,q)$-forms and for the Hodge $\ast$-operator as in \cite{ZT89,TZ07}.}
 
The deformation theory identifies the holomorphic tangent space $T_{\{E_{*}\}}\cN$ at a point $\{E_{*}\}\in\cN$ %corresponding to the equivalence class of a stable parabolic bundle $E_{*}$ 
with the complex vector space $\check{H}^{1}\left(\CC\PP^{1}, \mathrm{Par}\End E_{*}\right)$ modeling infinitesimal deformations of the parabolic bundle structure of a representative $E_{*}$, while the holomorphic cotangent space $T^{\ast}_{\{E_{*}\}}\cN$ is identified with the vector space $\check{H}^{0}\left(\CC\PP^{1}, \left(\mathrm{Par}\End E_{*}\right)^{\vee}\otimes K_{\CC\PP^{1}}\right)$ of parabolic Higgs fields on $E_{*}$.
The isomorphism of these vector spaces with spaces of square integrable $\End E_{0}$-valued harmonic forms on $X_{0}$ follows from Dolbeault's theorem and the structure of the bundle $\mathrm{Par}\End E_{*}$ of parabolic endomorphisms,\footnote{It is implicit in \cite{MS80} that for any admissible representation $\rho$, $\mathrm{Par}\End E^{\rho}_{*}=E^{\Ad\rho}$.} 
and as in the usual stable bundle case \cite{NS64},  $T_{\{E_{*}\}}\cN$ is also identified with $\cH^{0,1}(X_{0},\End E_{0})$ (see \cite{MS80} and \cite{TZ07} for details).
Correspondingly, $T^{\ast}_{\{E_{*}\}}\cN$ is identified with $\cH^{1,0}(X_{0},\End E_{0})$. The duality pairing 
$$\cH^{0,1}(X_{0},\End E_{0})\otimes\cH^{1,0}(X_{0},\End E_{0})\rightarrow\CC$$ 
is given by
\begin{equation}\label{eq:pairing}
(\nu,\theta)\mapsto\int\limits_{X_{0}}\nu\wedge\theta, \quad\nu
\in\cH^{0,1}(X_{0},\End E_{0}),\,\theta\in\cH^{1,0}(X_{0},\End E_{0}),
\end{equation}
where $\wedge$ denotes the composition of the wedge product of matrix-valued forms and the trace map $\tr: \End E_{0}\to\CC$.

\subsection{Automorphic forms of weight $2$ with the representation $\Ad\rho$} \label{Ad}

Let $\rho$ be an admissible representation of $\Ga$. By definition, a matrix-valued automorphic form of weight $2$ for the group $\Ga$ with the representation $\Ad\rho$ is a holomorphic $r\times r$ matrix-valued function
$f:\HH\rightarrow\End\CC^{r}$, satisfying
$$f(\s\tau)\s'(\tau)=\Ad\rho(\s)f(\tau)=\rho(\s)f(\tau)\rho(\s)^{-1},\quad\s\in\Ga.$$
An automorphic form is said to be regular if 
\begin{equation}\label{reg}
\lim_{\tau\rightarrow \infty}f(\sigma_{i}\tau)\sigma_{i}^{\prime}(\tau)
\end{equation}
exists for all $i=1,\dots,n$. 
Equivalently, since  $0 <\alpha_{i1}<\dots<\alpha_{ir}<1$,
\begin{equation} \label{F-Fourier}
f(\sigma_{i}\tau)\sigma_{i}^{\prime}(\tau)
=U_{i}q^{W_{i}}\left(\sum_{k=0}^{\infty}B_{i}(k)q^{k}\right)q^{-W_{i}}U_{i}^{-1},
\end{equation}
for every $i=1,\dots,n$, where the matrices $U_{i}\in \mathrm{SU}(r)$ satisfy
$\rho(S_{i})=U_{i}e^{2\pi\sqrt{-1}\,W_{i}}U_{i}^{-1}$, and the matrices $B_{i}(0)$ are lower triangular, i.e. $B_{i}(0)\in\frak{b}(r)$.
Denote by $\frak{M}_{2}(\Ga, \Ad\rho)$ the space of regular matrix-valued automorphic forms
of weight $2$ for $\Ga$ with the representation $\Ad\rho$. The subspace 
$$\frak{S}_{2}(\Ga, \Ad\rho)\subset\frak{M}_{2}(\Ga, \Ad\rho)$$ of cusp forms is defined by the conditions
$$\lim_{\tau\rightarrow \infty}f(\sigma_{i}\tau)\sigma_{i}^{\prime}(\tau)=0\quad\text{for all}\quad i=1,\dots,n,$$
or equivalently, by the matrices $B_{i}(0)$ being strictly lower triangular, $B_{i}(0)\in\frak{n}(r)$. 

The space 
$\frak{S}_{2}(\Ga, \Ad\rho)$ 
of cusp forms of weight 2 carries a natural inner product, the Petersson inner product,  given by the formula 
\begin{equation*}
\la f_{1}, f_{2}\ra =2\iint\limits_{D}\tr( f_{1}(\tau)f_{2}(\tau)^{\ast})d^{2}\tau,\quad f_{1}, f_{2}\in \frak{S}_{2}(\Ga,\Ad\rho),
\end{equation*} 
where $D$ is a fundamental domain of $\Gamma$ in $\HH$ and $d^{2}\tau = \frac{\sqrt{-1}}{2} d\tau\wedge d\bar{\tau}$. The integral is absolutely convergent when at least one of $f_{1}, f_{2}\in\frak{M}_{2}(\Ga, \Ad\rho)$ is a cusp form. 
There is a projection $P: \frak{M}_{2}(\Ga, \Ad\rho)\mapsto \frak{S}_{2}(\Ga, \Ad\rho)$, uniquely characterized by the property 
\[
\la P(f),g\ra=\la f,g\ra\quad\text{for all}\quad f\in\frak{M}_{2}(\Ga, \Ad\rho)\;\;\text{and}\;\; g\in  \frak{S}_{2}(\Ga,\Ad\rho).
\]
As an immediate consequence of the Mehta--Seshadri theorem, when $\rho$ is irreducible, there is an isometric isomorphism
$$\frak{S}_{2}(\Ga, \Ad\rho)\simeq T^{*}_{E^{\rho}_{*}}\cN.$$
Indeed, it follows from Lemma \ref{Theo-Upsilon} and the Fourier series expansions of $J(\tau)$ (see, e.g., \cite[Lemma 2]{ZT87a}) that for every $Y\in\Upsilon(\rho)$, the map
\begin{equation*} 
\frak{S}_{2}(\Ga, \Ad\rho)\ni f\mapsto \mathcal{F}\in \cH^{1,0}(X_{0}, \End E_{0}),
\end{equation*}
where
$$\mathcal{F}(z)= Y(J^{-1}(z))f(J^{-1}(z))(J^{-1})'(z)Y(J^{-1}(z))^{-1},$$
is an isomorphism.\footnote{The factor of 2 in the definition of the Petersson inner product reflects the normalization of the inner product of $(p,q)$-forms in \cite{ZT89,TZ07}.} 
The map $f\mapsto \mathcal{F}$ is also an isomorphism between the vector spaces of cusp forms and parabolic Higgs fields, realized in the affine trivialization of $E_{N}$ over $\CC$. Similarly, there is an isomorphism
$$\overline{\frak{S}_{2}(\Ga, \Ad\rho)}\simeq T_{E^{\rho}_{*}}\cN,$$
where $\overline{\frak{S}_{2}(\Ga, \Ad\rho)}$ is the vector space of Hermitian conjugates $f^{*}$ of $f\in\frak{S}_{2}(\Ga, \Ad\rho)$. The corresponding map
\begin{equation*}
\overline{\frak{S}_{2}(\Ga, \Ad\rho)}\ni f^{\ast}\mapsto \mathcal{F}^{*}\in \cH^{1,0}(X_{0}, \End E_{0})\simeq\check{H}^{1}\left(\CC\PP^{1}, \mathrm{Par}\End E^{\rho}_{*}\right)
\end{equation*}
is given by
$$ \mathcal{F}^{*}(z)= Y(J^{-1}(z))f^{\ast}(J^{-1}(z))\overline{(J^{-1})'}(z)Y(J^{-1}(z))^{-1}.$$
As a consequence of the above, the dimension formula  
$$\dim\frak{S}_{2}(\Ga,\Ad\rho)=\frac{1}{2}n(r^{2}-r)-r^{2}+1,$$
implying \eqref{dim}, can be obtained as a special case of the general formula in \cite[Corollary 4.2 on p.~485]{Hej83} (cf. \cite{Men17a}, where a proof is given in terms of the Riemann--Roch theorem for vector-valued automorphic forms).

\subsection{Automorphic forms for stable bundles} It is a classical result that for a Riemann surface $X_{0}$ of type $(0,n)$, the Schwarzian derivative of the uniformization map $J:\HH\mapsto X_{0}$ is a regular automorphic form of weight $4$ for $\Gamma$, which does not depend on a particular choice of $J$ in the orbit of the automorphism group $\mathrm{PSL}(2,\CC)$ of $X_{0}$ (see, e.g., \cite{ZT87a}). For a stable parabolic bundle $E^{\rho}_{*}$, the analog of the uniformization map is given by the function $Y(\tau)$ from Lemma \ref{Theo-Upsilon}, which realizes the isomorphism $E^{\rho}\cong E_{N}$,
and the analog of the Schwarzian derivative is given by the logarithmic derivative 
\begin{equation}\label{cA}
\cA(Y)(\tau)=-Y(\tau)^{-1}Y'(\tau).
\end{equation} 
It follows from Lemma \ref{Theo-Upsilon} that for every $Y\in\Upsilon(\rho)$, the function $\cA(Y)$ is an automorphic form of weight $2$ for the representation $\Ad\rho$, satisfying the regularity condition \eqref{reg} for $i=1,\dots,n-1$. For $i=n$ the function $\cA(Y)$ has asymptotics
\begin{equation}\label{A-n}
\cA(Y)(\sigma_{n}\tau)=2\pi\sqrt{-1}U_{n}(W_{n}+q^{W_{n}}C_{n}(0)^{-1}NC_{n}(0)q^{-W_{n}})U_{n}^{-1} +o(1)
\end{equation} 
as $\tau\rightarrow \infty$, which do not immediately guarantee its regularity at $i=n$ and hence that $\cA(Y)\in \frak{M}_{2}(\Ga, \Ad\rho)$. The automorphism group $\Aut E_{N}$ acts on the set $\frak{A}_{2}(\Ga,\Ad\rho) = \{\cA(Y)\; :\; Y\in\Upsilon(\rho)\}$ by the formula
\begin{equation}\label{action}
g\cdot\cA(Y)=\cA -Y^{-1}(g\circ J)^{-1}(g\circ J)'Y.
\end{equation}
It is a fundamental question whether for a given $Y\in\Upsilon(\rho)$ there is a unique $g\in \Aut E_{N}$ such that $g\cdot\cA(Y)\in\frak{M}_{2}(\Ga, \Ad\rho)$, and whether such a choice would depend continuously on moduli parameters. The existence and uniqueness of a regular orbit representative is equivalent to the solvability of the Riemann--Hilbert problem for $\rho$ (see \cite[Section 6.2]{Men18}).  However, such a choice could not be done consistently on the whole moduli space. The possibility of a maximal consistent choice determines a Zariski open subset $\cN_{0}\subseteq\cN$ in the open Harder--Narasimhan stratum of evenly-split stable parabolic bundles in $\cN$, on which there would be a unique $g\in \Aut E_{N_{0}}$ such that $g\cdot\cA(Y)\in\frak{M}_{2}(\Ga, \Ad\rho)$. We will call $\cN_{0}$ the \emph{regular locus}.

Namely, suppose that $N=N_{0}$ and consider first the simplest case $p=0$, so that $r\mid d$ and $N_{0}=m I_{r}$. In this case $\Aut E_{N_{0}}\cong\GL(r,\CC)$, 
so that by \eqref{action} the function $\cA(Y)$ is independent of the choice of $Y\in\Upsilon(\rho)$ and determines a unique automorphic form $\cA$ of weight $2$, and it follows from \eqref{A-n} that $\cA\in\frak{M}_{2}(\Ga, \Ad\rho)$. 

The case $0 < p < r$, i.e. $r\centernot\mid d$, is more subtle. In order to define the Zariski open subset $\cN_{0}\subseteq\cN$ where $\frak{A}_{2}(\Ga,\Ad\rho)$ contains exactly one regular automorphic form for each $\{E^{\rho}_{*}\}\in\cN_{0}$, we are required to impose an additional condition on the flags over $E_{\infty}$. %depending smoothly on moduli.
Namely, let $\cF(E_{\infty})$ be the complete flag manifold on $E_{\infty}$ and $\mathrm{Gr}_{r-p}(E_{\infty})$ be the Grassmannian of $(r-p)$-planes in $E_{\infty}$, together with the natural projection $\mathrm{pr}_{\infty}:\cF(E_{\infty})\rightarrow \mathrm{Gr}_{r-p}(E_{\infty})$. Restriction of the unique subbundle $\mathcal{O}(m+1)^{p}\hookrightarrow E$ to the fiber $E_{\infty}$ determines a special $p$-plane $V_{\infty}\subset E_{\infty}$. Denote by $\mathrm{Gr}^{0}_{r-p}(E_{\infty})\subsetneq \mathrm{Gr}_{r-p}(E_{\infty})$ the Zariski open subset consisting of $(r-p)$-planes $V'_{z}$ satisfying
\begin{equation}\label{Zariski}
V'_{\infty}\cap V_{\infty} = \{0\}.
\end{equation}
Under a choice of Mehta--Seshadri uniformization map $\cJ$, the $p$-plane $V_{\infty}\subset E_{\infty}$ gets identified with the span of $\{\mathbf{e}_{r-p+1},\dots,\mathbf{e}_{r}\}\subset\CC^{r}$, and the Zariski open condition \eqref{Zariski} is equivalent to the existence of a unique factorization $C_{n}(0) = M\Pi_{0}DL$, where $\Pi_{0}$ is the permutation matrix of the product of transpositions $(1,r)(2,r-1)\dots(\lfloor r/2\rfloor, r- \lfloor r/2\rfloor +1)$, and
\begin{equation}\label{eq:Bruhat}
M = \begin{pmatrix} I_{r-p} & 0\\
A & I_{p}\end{pmatrix},\quad
D = \begin{pmatrix} D_{p} & 0\\
0 & D_{r-p}\end{pmatrix},\quad  
L = \begin{pmatrix} I_{p} & 0\\
B & I_{r-p}\end{pmatrix},
\end{equation}
which is a consequence of the Bruhat decomposition for the group $\mathrm{GL}(r,\CC)$ (see \cite[Lemma 2, Remarks 4 \& 8]{Men18}), so that the product $DL$ belongs to the stabilizer of the span of $\{\mathbf{e}_{r-p+1},\dots,\mathbf{e}_{r}\}$. 

\begin{definition}\label{def:regular-locus} 
The regular locus $\cN_{0}\subseteq \cN$ is the set of isomorphism classes of evenly-split stable parabolic bundles $\left\{E_{*}\right\}$ whose flags at $\infty$ project onto $\mathrm{Gr}^{0}_{r-p}(E_{\infty})$.  
\end{definition}

In particular, when $r\mid d$, the second condition is vacuous and $\cN_{0}$ is just the open Harder--Narasimhan stratum of evenly-split parabolic bundles. 
The next couple of results and their consequences justify our definition of the regular locus.

\begin{lemma}\label{regular-auto}
Let $\rho:\Gamma\rightarrow\mathrm{U}(r)$ be a fixed admissible irreducible representation with $\left\{E^{\rho}_{*}\right\}\in\cN_{0}$. Then $\frak{A}_{2}(\Ga, \Ad\rho)\cap\frak{M}_{2}(\Ga, \Ad\rho)$ consists of a unique element $\cA$ depending smoothly on moduli, whose constant terms at the cusps when $r\mid d$ are
\begin{equation}\label{eq:regular-Fourier-1}
B_{i}(0)=\left\{
\begin{array}{cl}
2\pi\sqrt{-1}W_{i} &\qquad i=1,\cdots,n-1,\\\\
2\pi\sqrt{-1}(W_{n} + N_{0}) &\qquad i=n.
\end{array}
\right.
\end{equation}
and when $r\centernot\mid d$, 
\begin{equation}\label{eq:regular-Fourier-2}
B_{i}(0)=\left\{
\begin{array}{cl}
2\pi\sqrt{-1}W_{i} &\qquad i=1,\cdots,n-1,\\\\
2\pi\sqrt{-1}\left(W_{n} + \Ad\left(\Pi_{0}L\right)^{-1}(N_{0})\right) &\qquad i=n.
\end{array}
\right.
\end{equation}
where
$$\qquad L = \begin{pmatrix} I_{p} & 0\\
B & I_{r-p}\end{pmatrix},$$
for some $(r-p)\times p$ matrix $B$, and $\Pi_{0}$ is the permutation matrix of the product of transpositions $(1,r)(2,r-1)\dots(\lfloor r/2\rfloor, r- \lfloor r/2\rfloor +1)$.
\end{lemma}
\begin{proof}
Recall that the flag at $\infty$ is equivalently determined by the ordered frame of the matrix $C_{n}(0)$. The case $r\mid d$ has already been described, and the formulas \eqref{eq:regular-Fourier-1} are an immediate consequence of the Fourier series expansions \eqref{Y-Fourier}--\eqref{Y-Fourier-2} in Lemma \ref{Theo-Upsilon} for any choice $Y\in\Upsilon(\rho)$.  
When $r\centernot\mid d$, it follows from \eqref{A-n} that $\cA(Y)\in \frak{M}_{2}(\Ga, \Ad\rho)$ if and only if the matrix $\Ad C_{n}(0)^{-1}(N_{0})$ is lower triangular.
Consider the factorization $C_{n}(0)=M\Pi_{0}DL$ described before.
Since $\Pi_{0}D\Pi_{0}^{-1}$ is block-diagonal of block type $(r-p,p)$, it commutes with $N_{0}$. Consequently, $\Ad C_{n}(0)^{-1}(N_{0})$ is lower triangular if and only if $M = I_{r}$, in which case it equals  $\Ad\left(\Pi_{0}L\right)^{-1}(N_{0})$. It follows from \eqref{action} that the subgroup $\mathrm{P}_{N_{0}}\subset \Aut E_{N_{0}}$ acts trivially on $\frak{A}_{2}(\Ga, \Ad\rho)$, while the subgroup $\mathrm{N}_{N_{0}}$ acts by the transformations
$$M\mapsto g(z)\cdot M = \begin{pmatrix} I_{r-p} & 0\\
A + C& I_{p}\end{pmatrix},\quad g(z)\in\mathrm{N}_{N_{0}}.$$
In particular, there is a unique $g(z)\in\mathrm{N}_{N_{0}}$ such that $M = I_{r}$.\footnote{It follows that $\mathrm{Gr}^{0}_{r-p}(E_{\infty})$ is a principal homogeneous space for the subgroup $\mathrm{N}_{N_{0}}$ \cite[Corollary 1]{Men18}, and the block $A$ provides coordinates for it. The smooth dependence of the Riemann-Hilbert correspondence on moduli of irreducible admissible representations implies that the normalization $A = 0$ on $\cN_{0}$ would also depend smoothly on moduli.}  Formulas \eqref{eq:regular-Fourier-2} then follow from \eqref{Y-Fourier}--\eqref{Y-Fourier-2} and \eqref{A-n}, and the corresponding regular automorphic form $\cA(Y)$ depends smoothly on moduli.
\end{proof}

\begin{definition}
Let $\vartheta$ denote the $(1,0)$-form on $\cN_{0}$ defined pointwise at any given $\left\{E^{\rho}_{*}\right\}\in\cN_{0}$ as $\vartheta=P(\cA)\in \frak{S}_{2}(\Ga, \Ad\rho)$, i.e.
\begin{equation} \label{Theta-form}
\vartheta_{E^{\rho}_{*}}(\nu)=\langle\mathscr{A},\nu^{*}\rangle=2\iint\limits_{D}\text{tr}(\mathscr{A}\nu)d^{2}\tau, \;\;\;\nu\in \overline{\frak{S}_{2}(\Ga, \Ad\rho)}.
\end{equation}
\end{definition}

\begin{lemma}\label{special-Y}
For any given $\{E^{\rho}_{*}\}\in\cN_{0}$, let $\cA$ be normalized as in Lemma \ref{regular-auto}. Then there is a unique representative $Y\in\Upsilon(\rho)$ satisfying $\cA(Y) = \cA$, and such that when $r\mid d$, the Fourier series expansions \eqref{Y-Fourier-2} take the form
\begin{equation}\label{eq:Y-1}
Y(\sigma_{n}\tau) = \left(I_{r} +\sum_{k=1}^{\infty}C_{n}(k)q^{k}\right)q^{-(W_{n}+N_{0})}U_{n}^{-1},
\end{equation}
and when $r\centernot\mid d$, 
\begin{equation}\label{eq:Y-2}
Y(\sigma_{n}\tau) = \left(\Pi_{0} +\sum_{k=1}^{\infty}C_{n}(k)q^{k}\right)q^{-(W_{n}+N'_{0})}U_{n}^{-1},
\end{equation}
where $N'_{0} = \Ad(\Pi_{0})^{-1}(N_{0})$.
\end{lemma}
\begin{proof}
When $r\mid d$, there is a unique normalization $C_{n}(0) = I$, and the result follows since $N_{0}$ is a multiple of the identity. 

Assume now that $r\centernot\mid d$. Using $q^{-N_{0}}=\Pi_{0}q^{-N'_{0}}\Pi_{0}^{-1}$, we rewrite the Fourier expansion \eqref{Y-Fourier-2} as
\begin{equation} \label{Y-Fourier-3} 
Y(\sigma_{n}\tau) =\sum_{k=0}^{\infty}\Pi_{0}q^{-N_{0}'}\Pi_{0}^{-1}C_{n}(k)q^{k}\;q^{-W_{n}}U_{n}^{-1}.
\end{equation}
Since $N'_{0}=\mathrm{diag}(\underbrace{m+1,\dots,m+1}_{p},\underbrace{m,\dots,m}_{r-p})$, we have 
\begin{equation}\label{formula-0}
q^{-N_{0}'} \begin{pmatrix} A & B\\
C & D\end{pmatrix}=\begin{pmatrix} A & q^{-1}B\\
qC & D\end{pmatrix}q^{-N_{0}'}
\end{equation}
for any block $(p,r-p)$ matrix $\begin{pmatrix} A & B\\
C & D\end{pmatrix}$. Using this formula, we can move $q^{-N_{0}'}$ to the right of the matrices $C_{n}(k)$ in \eqref{Y-Fourier-3} and get a new expansion
\begin{equation*} 
Y(\sigma_{n}\tau) =\sum_{k=0}^{\infty}C'_{n}(k)q^{k}\;q^{-(W_{n}+N_{0}')}U_{n}^{-1}.
\end{equation*}  
According to Lemma \ref{regular-auto}, an $\mathrm{N}_{N_{0}}$-orbit of $C_{n}(0)$ contains a unique representative $\Pi_{0}DL$, so that 
$$C'_{n}(0)=\Pi_{0}\begin{pmatrix} D_{p} & B'\\
0 & D_{r-p}\end{pmatrix},$$
where the $p\times(r- p)$ block $B'$ comes from the corresponding $p\times(r- p)$ block in the matrix $\Pi_{0}^{-1}C_{n}(1)$.
Since the remaining group of bundle automorphisms $\mathrm{P}_{N_{0}}\subset \Aut E_{N_{0}}$ consists of block $(r-p,p)$ matrices of the form
\[
g = \begin{pmatrix} D'_{r-p} & 0\\
C' & D'_{p}\end{pmatrix},
\]
it follows that there is a unique $g\in\mathrm{P}_{N_{0}}$ such that $C'_{n}(0)=\Pi_{0}$.
\end{proof}

\subsection{Complex coordinates and K\"{a}hler form} \label{Kahler}
Although we only consider the case of $\CC\PP^{1}$, the subsequent results are valid for any compact Riemann surface $X$ with Fuchsian model $X_{0} = X\setminus\{z_{1},\dots,z_{n}\}\cong \Gamma\setminus \HH$, and were developed in \cite{ZT89,TZ07}. Namely, let $\rho:\Gamma\rightarrow \mathrm{U}(r)$ be an admissible irreducible representation. We have the following result. 

\begin{proposition} \label{prop-def}
For each $\nu\in\overline{\mathfrak{S}_{2}(\Gamma,\text{\emph{Ad}}\,\rho)}$ and $\varepsilon\in\mathbb{C}$ sufficiently close to $0$, there is a unique solution $f^{\varepsilon\nu}:\mathbb{H}\rightarrow\text{\emph{GL}}(r,\mathbb{C})$ of the differential equation
\begin{equation}\label{quasiconformal}
f^{-1} f_{\overline{\tau}}=\varepsilon\nu
\end{equation}
with the following properties.
\renewcommand{\theenumi}{\roman{enumi}}
\begin{enumerate}
\item $f^{\varepsilon\nu}(\gamma\tau)=\rho^{\varepsilon\nu}(\gamma)f(\tau)\rho(\gamma)^{-1}\quad \forall\gamma\in\Gamma$, where $\rho^{\varepsilon\nu}:\Gamma\rightarrow\text{\emph{U}}(r)$ is an admissible irreducible representation;
\item $\det f^{\varepsilon\nu}(\tau_{0})=1$ for some fixed $\tau_{0}\in\mathbb{H}$;
\item $f^{\varepsilon\nu}$ is regular at the cusps, that is
\[
\lim_{\tau\rightarrow \tau_{i}}f^{\varepsilon\nu}(\tau)<\infty,\qquad i=1,\dots,n.
\]
\end{enumerate}
\renewcommand{\theenumi}{\arabic{enumi}}
\end{proposition}

$f^{\varepsilon\nu}$ is real analytic in $\varepsilon$, and  is analogous to a corresponding quasiconformal mapping in Teichm\"{u}ller theory (see \cite{ZT87a}). Under the special choice of bundle uniformization maps $Y_{\varepsilon\nu}$ following from Lemma \ref{special-Y}, it defines a parabolic bundle map $F^{\varepsilon\nu}:=\left(Y_{\varepsilon\nu} f^{\varepsilon\nu} Y^{-1}\right)\circ J^{-1}$ by requiring the commutativity of the diagram of parabolic bundles
\begin{equation}\label{com-diagram}
\xymatrix{
\mathbb{H}^{*}\times\mathbb{C}^r_{*} \ar[r]^{f^{\varepsilon\nu}} \ar[d]_{\mathscr{J}}& \mathbb{H}^{*}\times\mathbb{C}^r_{*} \ar[d]^{\mathscr{J}_{\varepsilon\nu}}\\
E^{\rho}_{*} \ar[r]^{F^{\varepsilon\nu}}  & E^{\rho^{\varepsilon\nu}}_{*}
} 
\end{equation}
It follows from Lemma \ref{special-Y} that $\det F^{\varepsilon\nu}(J(\tau_{0})) = 1$, since $(\det Y_{\varepsilon\nu}/\det Y)\circ J$ is a holomorphic function on $\CC\PP^{1}$ whose value at $\infty$ is equal to 1.

Given a basis $\nu_{1},\dots,\nu_{d}$ for $\cH^{0,1}(X,\End E_{0}^{\rho})$, let $\nu=\vep_{1}\nu_{1}+\dots+\vep_{d}\nu_{d}$, with $\vep_{i}\in\CC,\; i=1,\dots,d,$ be sufficiently small.
The induced mapping $(\vep_{1},\dots,\vep_{d})\mapsto \{E^{\rho^{\nu}}_{*}\}$ provides a coordinate chart on $\mathscr{N}$ in the neighborhood of the point $\{E^{\rho}_{*}\}$.
These coordinates transform holomorphically and endow $\mathscr{N}$ with the structure of a complex manifold (in direct analogy to Bers' coordinates on Teichm\"{u}ller spaces). The differential of such coordinate transformations is the linear mapping $\cH^{0,1}(X,\End E_{0}^{\rho})\rightarrow\cH^{0,1}\left(X,\End E_{0}^{\rho^{\nu}}\right)$ explicitly given by the formula
\begin{equation} \label{differential}
\mu\mapsto P_{\nu}(\Ad f^{\nu}(\mu)),\quad\mu\in\cH^{0,1}(X,\End E_{0}^{\rho}).
\end{equation}
Here $P_{\nu}$ is the orthogonal projection onto $\cH^{0,1}\left(X,\End E_{0}^{\rho^{\nu}}\right)$, while $\Ad f^{\nu}(\mu):=f^{\nu}\mu (f^{\nu})^{-1}$ is understood as a fiberwise linear mapping $\End E_{0}^{\rho}\rightarrow\End E_{0}^{\rho^{\nu}}$ (see \cite{TZ07} for details). The moduli space $\cN$ carries a Hermitian metric given by the Petersson inner product and the isomorphism $T_{\{E^{\rho}_{*}\}}\cN
\simeq\overline{\mathfrak{S}_{2}(\Gamma,\Ad\rho)}$.
This metric is analogous to the Weil-Petersson metric on Teichm\"{u}ller space, and for the moduli spaces of stable bundles of fixed rank and degree was introduced in \cite{Nar70, AB83}. This metric is K\"{a}hler \cite{TZ07} and we denote its K\"{a}hler form by $\Omega$:
$$\Omega\left(\frac{\pa}{\pa \vep(\mu)},\frac{\pa}{\pa\overline{ \vep(\nu)}}\right)=\frac{\sqrt{-1}}{2}\langle\mu,\nu\rangle.$$
Here $\displaystyle{\frac{\pa}{\pa \vep(\mu)}}$ and $\displaystyle{\frac{\pa}{\pa\overline{ \vep(\nu)}}}$ are the holomorphic and antiholomorphic tangent vectors at  $\{E_{*}\}\in\cN$ corresponding to $\mu,\nu\in\cH^{0,1}(X,\End E_{0})$ respectively.

\subsection{Variational formulas} \label{Var}
Here we collect the necessary variational formulas. Except for Lemma \ref{var-F}, these formulas are proved in \cite{ZT89,TZ07}. For $\nu\in T_{\{E^{\rho}_{*}\}}\cN
\simeq\overline{\mathfrak{S}_{2}(\Gamma,\Ad\rho)}$ put $$\dot{f}_{+}=\left.\frac{\partial f^{\varepsilon\nu}}{\partial\varepsilon}\right|_{\varepsilon=0}\quad\text{and}
\quad\dot{f}_{-}=\left.\frac{\partial f^{\varepsilon\nu}}{\partial\bar{\varepsilon}}\right|_{\varepsilon=0}.$$ 
\begin{lemma}[Vanishing of the first variation of the Hermitian metric]\label{Ahlfors}
For $\nu\in\overline{\mathfrak{S}_{2}(\Gamma,\Ad\rho)}$ we have
\begin{equation} \label{h-first}
\left.\frac{\partial}{\partial\varepsilon}\left((f^{\varepsilon\nu})^{*}f^{\varepsilon\nu}\right)\right|_{\varepsilon=0}
=\left.\frac{\partial}{\partial\bar{\varepsilon}}\left((f^{\varepsilon\nu})^{*}f^{\varepsilon\nu}\right)\right|_{\varepsilon=0}=0
\end{equation}
and also
\begin{equation}\label{var-formulas}
\left(\dot{f}_{+}\right)_{\bar{\tau}}=\nu,\qquad
\left(\dot{f}_{-}\right)_{\bar{\tau}}=0,
\end{equation}
\begin{equation} \label{var-formulas-2}
\quad\,\left(\dot{f}_{+}\right)_{\tau}=0,\qquad \left( \dot{f}_{-}\right)_{\tau}=-\nu^{*}.
\end{equation}
\end{lemma}

Let $\cC\to \cN$ be the holomorphic affine bundle modeled on $T^{*}\cN$, whose fiber over a given $\{E_{*}\}\in\cN$ corresponds to the affine space $\cC(E_{*})$. There is a $\bar{\partial}$-operator on the space of smooth sections $\Gamma\left(\cN,\cC\right)$ and taking values in $\Omega^{1,1}(\cN)$, defining the notion of holomorphicity of a global section. The theorem of Mehta-Seshadri provides a special section $s_{\mathrm{MS}}$, which we call the \emph{Mehta--Seshadri section}, given by logarithmic connections with irreducible unitary monodromy. The section $s_{\mathrm{MS}}$ is not holomorphic. Its non-holomorphicity is measured by the Narasimhan--Atiyah--Bott K\"ahler form. The next proposition is an analog of  \cite[Theorem 1]{ZT89} for parabolic bundles. As in \cite{ZT89}, its proof can be deduced from the variational formulas \eqref{var-formulas-2}. 

\begin{proposition}\label{prop:MS-section}
For any moduli space $\cN$ of stable parabolic bundles, the Mehta--Seshadri section $s_{\mathrm{MS}}: \cN \to \cC$ satisfies 
\[
\bar{\partial} s_{\mathrm{MS}} = - 2\Omega.
\]
\end{proposition}

For a given $\{E^{\rho}_{*}\}\in\cN_{0}$, $\nu\in\overline{\mathfrak{S}_{2}(\Gamma,\text{Ad}\,\rho)}$, and $\varepsilon$ sufficiently small, consider 
the families of normalized maps $Y_{\varepsilon\nu}$ as in Lemma \ref{special-Y}, and the parabolic bundle maps $F^{\varepsilon\nu}: E^{\rho}_{*}\rightarrow E^{\rho^{\varepsilon\nu}}_{*}$. 
Put $$\dot{Y}_{+}=\left.\frac{\partial Y_{\varepsilon\nu}}{\partial\varepsilon}\right|_{\varepsilon=0},\quad \dot{Y}_{-}=\left.\frac{\partial Y_{\varepsilon\nu}}{\partial\bar{\varepsilon}}\right|_{\varepsilon=0}
,$$ and 
\begin{equation}\label{Eichler+-}
\displaystyle\mathscr{E}_{+}=Y^{-1}\dot{Y}_{+},\qquad 
\displaystyle\mathscr{E}_{-}=Y^{-1}\dot{Y}_{-}.
\end{equation}
Denote by $h^{\varepsilon\nu}= (\mathscr{Y}_{\varepsilon\nu}\mathscr{Y}_{\varepsilon\nu}^{*})^{-1}$ the local family of Hermitian metrics on the bundles $E^{\rho^{\varepsilon\nu}}$, where $\mathscr{Y}_{\varepsilon\nu}=Y_{\varepsilon\nu}\circ J^{-1}$.

\begin{lemma}\label{var-F} Over the regular locus $\mathscr{N}_{0}$, we have 
\begin{equation} \label{F-var}
F^{\varepsilon\nu}=I_{r}+\varepsilon\dot{F}+o(|\varepsilon|),
\end{equation}
where
\begin{equation}\label{h-F-1}
\dot{F}=-h^{-1}\dot{h}_{+} = -h^{-1}\left.\frac{\partial h^{\varepsilon\nu}}{\partial\varepsilon}\right|_{\varepsilon=0},
\end{equation}
is a smooth endomorphism of $E_{N_{0}}$ satisfying 
\begin{equation}\label{relation-F-nu}
(\dot{F}\circ J)_{\bar{\tau}}=Y\nu Y^{-1}.
\end{equation}
\end{lemma}
\begin{proof} Consider the equation
\begin{equation*}
Y_{\varepsilon\nu}\cdot f^{\varepsilon\nu}=(F^{\varepsilon\nu}\circ J)\cdot Y,
\end{equation*}
which follows from \eqref{com-diagram}, and differentiate it with respect to $\varepsilon$ and $\bar{\varepsilon}$ at $\varepsilon=0$. We obtain
\begin{align}\label{rel-Eichler-f1}
Y\left(\mathscr{E}_{+}+\dot{f}_{+}\right) Y^{-1}&=\dot{F}_{+}\circ J,\\
\intertext{and}
Y\left(\mathscr{E}_{-}+\dot{f}_{-}\right) Y^{-1}&=\dot{F}_{-}\circ J. \label{rel-Eichler-f2}
\end{align}
It follows from \eqref{var-formulas} that the left hand side of \eqref{rel-Eichler-f2} is holomorphic in $\tau$.  Moreover, the function $\mathscr{E}_{-}+\dot{f}_{-}=Y^{-1}(\dot{F}_{-}\circ J)Y$ is $\mathrm{Ad}\,\rho$-automorphic and is bounded at the cusps, i.e., is a parabolic endomorphism of $E^{\rho}_{*}$. Therefore, $\cE_{-}+\dot{f}_{-}=cI=\dot{F}_{-}$ and it folows from the normalization of $F^{\varepsilon\nu}$ that indeed $\dot{F}_{-}=0$, which gives \eqref{F-var} with $\dot{F}=\dot{F}_{+}$.
 
To prove \eqref{h-F-1}, consider the equation
$$(Y^{-1})^{*}(f^{\varepsilon\nu})^{*}f^{\varepsilon\nu}
Y^{-1}=\left((F^{\varepsilon\nu})^{*}h^{\varepsilon\nu}F^{\varepsilon\nu}\right)\circ J,$$
differentiate it with respect to $\varepsilon$ at $\varepsilon=0$ and use Lemma \ref{Ahlfors}.
Finally, \eqref{relation-F-nu} immediately follows from \eqref{rel-Eichler-f1} and  \eqref{var-formulas}.
\end{proof}

For each pair of harmonic forms $\mu,\nu\in \cH^{0,1}\left(X_{0},\End E^{\rho}_{0}\right)$ we define a smooth $L^{2}$-section $f_{\mu\overline{\nu}}$ of $\End E^{\rho}_{0}$
 by the formula
\begin{equation}\label{eq:def-f-mu-nu}
f_{\mu\overline{\nu}} = \Delta_{0}^{-1} \left(* [*\mu,\nu]\right),
\end{equation}
where $\Delta_{0}$ is the restriction of the Laplace operator $\Delta$ to the orthogonal complement of the identity endomorphism (the kernel of $\Delta$) in the Hilbert space of $L^{2}$-sections of $\End E^{\rho}_{0}$ (see \cite{TZ07} for details). 
The family
\[
\mu^{\varepsilon\nu} = \mathrm{P}_{\varepsilon\nu}\left( \Ad f^{\varepsilon\nu} (\mu)\right),
\]
for sufficiently small $\vep$, determines the tangent vector field $\partial/\partial\varepsilon(\mu)$ to the complex curve $\left\{E^{\rho^{\varepsilon\nu}}\right\}\subset \cN$. Let 
\[
L_{\bar{\nu}}\mu  =\left.\frac{\partial }{\partial\bar{\varepsilon}}\right|_{\varepsilon=0} \Ad\left(f^{\varepsilon\nu}\right)^{-1}\left(\mu^{\varepsilon\nu}\right)
\]
be the corresponding Lie derivative. We have (see \cite{TZ07})
\begin{equation}\label{eq:Lie-derivative}
L_{\bar{\nu}}\mu = \bar{\partial} f_{\mu\bar{\nu}}.
\end{equation}

\subsection{Tautological line bundles over $\cN$}\label{taut}
For each marked point $z_{i}\in\CC\PP^{1}$ and $j=1,\dots,r$, let  $\ell_{ij}$ be the holomorphic line bundle over $\cN$ whose fiber over $\{E_{*}\}\in\cN$ is the complex line $E_{ij}/E_{ij+1}$.
The isomorphism $E_{*} \cong E^{\rho}_{*}$ identifies such fiber with the complex line $L_{ij}$, the eigenspace for the eigenvalue $e^{2\pi\sqrt{-1}\alpha_{ij}}$ of $\rho(S_{i})$, and the Hermitian metric on $L_{ij}$, induced by the standard Euclidean metric on $\CC^{r}$, determines a natural Hermitian metric $\|\cdot\|_{ij}$ on $\ell_{ij}$. The line bundles $\ell_{ij}$ over $\cN$ are called \emph{tautological line bundles}.

Let $\Omega_{ij} = c_{1}(\ell_{ij}, \|\cdot\|_{ij})$\footnote{In \cite{TZ07}, the normalization $c_{1}(\ell_{ij}, \|\cdot\|_{ij}) = \dfrac{2}{\pi}\Omega_{ij}$ is used instead.} be the first Chern form of the tautological line bundle $\ell_{ij}$ with respect to the Hermitian metric $\|\cdot\|_{ij}$.
It was proved in \cite[Lemma 4]{TZ07}) that
\begin{equation}\label{eq:cusp-Chern-form}
\Omega_{ij}\left(\frac{\partial}{\partial\varepsilon(\mu)}, \frac{\partial}{\partial\overline{\varepsilon(\nu)}}\right) = \frac{\sqrt{-1}}{2\pi}\tr\left(F^{i}_{\mu\overline{\nu}} v_{ij}\right),
\end{equation}
where $F^{i}_{\mu\overline{\nu}}\in\End\CC^{r}$ are the leading terms of the asymptotics of $f_{\mu\overline{\nu}}(\sigma_{i}\tau)$ at the cusps,
\begin{equation}\label{eq:asym-f-mu-nu}
f_{\mu\overline{\nu}}(\sigma_{i}\tau) = F^{i}_{\mu\overline{\nu}} + o(1)\quad \text{as}\quad \im\tau\to \infty,
\end{equation}
and
\begin{equation}\label{eq:u-v}
v_{ij} = u_{ij}\otimes \bar{u}_{ij} - I_{r}/r\in \End \CC^{r},\qquad j=1,\dots,r,
\end{equation}
where $\rho(S_{i})u_{ij} = e^{2\pi\sqrt{-1}\alpha_{ij}}u_{ij}$, $\|u_{ij}\| = 1$, and $\bar{u}_{ij}$ stands for the componentwise complex conjugate of $u_{ij}$. The matrices $v_{ij}$ span the subspace of traceless matrices in $\ker\left(\Ad\rho\left(S_{i}\right) - I_{r}\right)\subset \End\CC^{r}$, and the matrices $F^{i}_{\mu\bar\nu}$ are traceless and satisfy
$$\tr\left(F^{i}_{\mu\bar\nu}v\right) = 0\quad v\notin\ker\left(\Ad\rho\left(S_{i}\right) - I_{r}\right).$$ 
(See \cite{TZ07} for details and the relation of the $(1,1)$-forms $\Omega_{ij}$ on $\cN$ with the matrix-valued Eisenstein-Maass series.)

\section{The WZNW action} \label{WZ}

\subsection{Zero curvature equation and WZNW functional}
Given an open set $\cU\subseteq\CC$, let $E = \cU\times \CC^{r}$ be a holomorphically trivial vector bundle of rank $r$. A Hermitian metric on $E$ is then a smooth map $h:\cU\to\CMcal{H}_{r}$, where $\CMcal{H}_{r}$ is the homogeneous space of Hermitian and positive-definite $r\times r$ matrices. The zero-curvature equation for the corresponding Chern connection $d+h^{-1}\partial h$ has the form 
\begin{equation} \label{0-curv}
\bar{\del}(h^{-1}\del h)=0.
\end{equation}
Equation \eqref{0-curv} is the Euler-Lagrange equation for the Wess-Zumino-Novikov-Witten (WZNW) functional, introduced by Novikov \cite{N82} and Witten \cite{W84} in the case when the target space is a compact simply-connected Lie group $G$. Namely, for smooth maps $g:\CC\PP^{1}\to G$ consider the functional
$$S[g]=S_{0}[g]+W[g],$$
where $$S_{0}[g]=\frac{\sqrt{-1}}{2}\iint\limits_{\CC\PP^{1}}\tr\left(g^{-1}\del g \wedge g^{-1}\bar{\del} g\right),$$ and $$W[g]=\frac{1}{6\sqrt{-1}}\int_{B_g}\Theta,$$ where $B_g$ is any smooth 3-chain in $G$ with the boundary $g(\CC\PP^{1})$.  Here $\Theta$ is the invariant 3-form on $G$,
\[
\Theta=\text{tr}(\theta\wedge\theta\wedge\theta), 
\]
where $\theta=g^{-1}dg$ is the  Maurer-Cartan form.\footnote{ The form $\Theta$ is closed and $\frac{1}{48\pi^{2}}[\Theta]$ is a generator of $H^{3}(G,\ZZ)$.}
The difference between any two choices $B_{g}-B'_{g}$ is a 3-cycle in $G$, and since
\[
\int_{B_{g}-B'_{g}}\Theta\in48\pi^{2}\mathbb{Z},
\]
it follows that $S[g]$ is defined only modulo $8\pi^{2}\sqrt{-1}\,\ZZ$. Though the action functional $S[g]$ is not single-valued, its Euler-Lagrange equation is well-defined, and is given by the corresponding analog of equation \eqref{0-curv}. 

In physics, the action functional $S[g]$ determines the so-called WZNW model, a two-dimensional nonlinear sigma-model with target space $G$. $S_{0}[g]$ plays the role of kinetic energy and $W[g]$ is the celebrated Wess-Zumino topological term (see \cite{DF} for details). 
In turn, in the present situation, equation \eqref{0-curv} is the Euler-Lagrange equation for a non-compact WZNW model \cite{Gaw92}, where the target is instead the homogeneous space $\CMcal{H}_{r}\cong\mathrm{GL}(r,\CC)/\mathrm{U}(r)$ for the non-compact Lie group $\mathrm{GL}(r,\CC)$, and whose action functional will now be described in detail.

\subsection{Cholesky decomposition} \label{Chol}
The homogeneous space $\mathcal{H}_{r}$ has natural global coordinates given by the so-called \emph{Cholesky decomposition}. Namely, every $h\in\CMcal{H}_{r}$ admits a unique decomposition $h=b^{*}b$, where $b$ is upper triangular with positive diagonal elements. The coordinates of $\CMcal{H}_{r}$ are then given by the matrix elements of $b$. Equivalently, the Cholesky decomposition can be written as $h=c^* a c$, where $a$ is diagonal and positive, and $c$ is unipotent, in such a way that $b=\sqrt{a}c$. 

Explicitly, represent a matrix $h\in\mathcal{H}_r$ as $h=M^*M$ for some $M\in\GL(r,\CC)$ and denote by $M^{l_{1}\cdots l_{j}}_{l'_{1}\cdots l'_{j}}$ the $j\times j$ matrix obtained from $M$ by selecting $l_{1},\dots,l_{j}$ rows and  $l'_{1},\dots,l'_{j}$ columns. Then
\begin{equation}\label{Cholesky-a-c}
a_{j}=\frac{Q_{jj}}{Q_{j-1j-1}},\qquad c_{jk}=\frac{Q_{jk}}{Q_{jj}},
\end{equation}
where
\begin{equation} \label{Q}
Q_{jk}=\sum_{l_{1}<\cdots< l_{j}}\overline{\det\left(M^{l_{1}\cdots l_{j}}_{1\cdots j-1 j}\right)}\det\left(M_{1\cdots j-1 k}^{l_{1}\cdots l_{j}}\right),
\end{equation}
which is essentially the Gram-Schmidt orthogonalization process. 

Since $\mathcal{H}_r$ is contractible, the restriction of the $3$-form $\Theta$ to $\mathcal{H}_{r}$ (which we continue to denote by $\Theta$) is exact.
It turns out that its primitive can be written down explicitly in terms of the Cholesky decomposition.
\begin{lemma} \label{Theta} On $\mathcal{H}_r$ we have
\begin{equation}\label{Theta-0}
\Theta=3d\tr(\theta_{1}\wedge\theta_{1}^{*}),
\end{equation}
where $\theta_{1}=d b \,b^{-1}$.
\end{lemma}
\begin{proof} Since $\theta=h^{-1}dh = b^{-1}((b^{*})^{-1}db^{*} + db\,b^{-1})b$ we get
\begin{align*}
\Theta &=\tr(\theta_1\wedge\theta_1\wedge\theta_1+3\theta_1\wedge\theta_1\wedge\theta_1^* + 3\theta_1^{*}\wedge\theta_1^*\wedge\theta_1 +\theta_1^*\wedge\theta_1^*\wedge\theta_1^*)\\
&=3\tr(\theta_1\wedge\theta_1\wedge\theta_1^* + \theta_1^{*}\wedge\theta_1^*\wedge\theta_1).
\end{align*}
Here we have used that since $b=\sqrt{a}c$ with diagonal $a$ and unipotent $c$,
$$\tr(\theta_1\wedge\theta_1\wedge\theta_1)=\frac{1}{8}\tr(da\wedge da\wedge da)=0$$
and similarly, $\tr(\theta_1^*\wedge\theta_1^*\wedge\theta_1^*)=0$. Using the Maurer-Cartan equations $d\theta_1=\theta_1\wedge\theta_1$
and $d\theta_1^*=-\theta_1^*\wedge\theta_1^*$, we obtain \eqref{Theta-0}.
\end{proof}

\subsection{The WZNW functional for singular metrics} \label{WZ action}

Using a choice of Birkhoff--Grothendieck splitting $E \cong E_{N}$ of a stable parabolic bundle $E_{\ast}$ on $\CC\PP^{1}$,  one can think of a singular Hermitian metric $h_{E}$ adapted to a parabolic structure as a smooth map $h: X_{0} \to \CMcal{H}_{r}$ satisfying \eqref{0-curv} and having prescribed asymptotic behavior at $z_{1},\dots, z_{n}$. 

Namely, it follows from  \eqref{Y-1} that
\begin{equation} \label{as-1}
h_{E}(z,\bar{z})=
G_{i}(z)^{*}|z-z_{i}|^{2W_{i}}G_{i}(z),
\end{equation}
in a neighborhood of $z_{i}$, $i=1,\dots,n-1$, where $G_{i}(z)$ are holomorphic and such that $G_{i}(z_{i})=C_{i}(0)^{-1}$. When $\{E_{*}\}\in\cN_{0}$, the corresponding $h_{E}$ in given by \eqref{h}, where the map $Y$ is normalized as in Lemma \ref{special-Y}. Then it follows from formulas \eqref{eq:Y-1}--\eqref{eq:Y-2} that in a neighborhood of  $\infty$,
\begin{equation}\label{as-2}
h_{E}(z,\bar{z})=G_{n}(z)^{*}|z|^{-2W'_{n}}G_{n}(z), 
\end{equation}
where $G_{n}(z)$ is holomorphic at $\infty$, $G_{n}(\infty) = I_{r}$, and
\[ 
W'_{n}=\left\{
\begin{array}{cl}
W_{n}+N_{0} & \text{if} \quad r\mid d\\\\
\Ad(\Pi_{0})(W_{n})+N_{0} & \text{if} \quad r\centernot\mid d
\end{array}
\right.
\]
Consider the class of functions $\cL_{z_{i}}$, defined on a neighborhood of each $z_{i}$ by the  following uniformly convergent series
\[
f_{i0}(z)g_{i0}(\bar{z})+\sum_{l=1}^{\infty}f_{il}(z)g_{il}(\bar{z})|z-z_{i}|^{\lambda_{il}},
\]
where the functions $f_{il}$ are holomorphic, $g_{il}$ are antiholomorphic, and $\lambda_{il}>0$ for $i=1,\dots,n-1$ and $\lambda_{nl}<0$. Then the functions $Q_{jk}(z,\bar{z})$, $1\leq j<k\leq r$, induced from \eqref{Q}, have the following form in a neighborhood of each $z_{1},\dots,z_{n}$: 
\[
Q_{jk}(z,\bar{z})=|z-z_{i}|^{2\kappa_{ij}}p_{ijk}(z,\bar{z}),
\]
where $\kappa_{ij}=\sum_{l=1}^{j}\alpha_{il}$ for $i=1,\dots,n-1$, $\kappa_{nj}=-\sum_{l=1}^{j}\alpha'_{il}$, and $p_{ijk}\in\cL_{z_{i}}$.
It follows from \eqref{Cholesky-a-c} that the functions $a_{j}$ and $c_{jk}$ have the following asymptotics in a neighborhood of $z_{1},\dots,z_{n}$:  
\begin{equation}\label{as-a}
a_{j}(z,\bar{z})=
\begin{cases}
|z-z_{i}|^{2\alpha_{ij}}\displaystyle\frac{p_{ijj}(z,\bar{z})}{p_{ij-1j-1}(z,\bar{z})} & i=1,\dots,n-1,\\\\
|z|^{-2\alpha'_{nj}}\displaystyle\frac{p_{njj}(z,\bar{z})}{p_{nj-1j-1}(z,\bar{z})} & i=n,
\end{cases}
\end{equation}
and
\begin{equation}\label{as-c}
c_{jk}(z,\bar{z})=\frac{p_{ijk}(z,\bar{z})}{p_{ijj}(z,\bar{z})}.
\end{equation}

For any given $\{E_{*}\}\in\cN_{0}$, let us denote by $\cC(E_{*}; X_{0},\CMcal{H}_{r})$ the space of smooth maps $h:X_{0}\rightarrow \CMcal{H}_{r}$ with asymptotics \eqref{as-1}--\eqref{as-2} for some local smooth functions $G_{1},\dots, G_{n}$, such that
\[
[M_{i}]= [C_{i}(0)]\in\GL(r,\CC)/\mathrm{B}(r),\quad i=1,\dots,n-1,
\]
where $M_{i} = G_{i}(z_{i})^{-1}$, and $G_{n}(\infty)=I_{r}$. We will now define the WZNW functional $S:\cC(E_{\ast}; X_{0},\mathcal{H}_{r})\to\RR$, whose unique critical point will be the singular Hermitian metric $h_{E}$ on $E_{*}$. Namely, for $\delta>0$ let
\[
X_{\delta}=\mathbb{C}\setminus\left(\bigcup_{i=1}^{n-1}\{|z-z_{i}|<\delta\}\cup\left\{|z|>\frac{1}{\delta}\right\}\right)
\]
with boundary $\partial X_{\delta} = \bigcup_{i = 1}^{n} C_{i\delta}$. Each component $C_{i\delta}$ carries a naturally induced orientation. Define

\begin{align*}
S_{0\delta}[h] & =  \frac{\sqrt{-1}}{2}\iint\limits_{X_{\delta}}\tr\left(h^{-1}\partial h\wedge h^{-1}\bar{\partial} h\right)\\ 
& + \frac{\sqrt{-1}}{2}\sum_{i = 1}^{n - 1}\int\limits_{C_{i\delta}}\tr\left(W_{i}|z-z_{i}|^{-2W_{i}}M_{i}^{*}hM_{i}\right)\displaystyle\left(\frac{dz}{z - z_{i}} - \frac{d\bar{z}}{\bar{z} - \bar{z}_{i}}\right)\\
& +   \frac{\sqrt{-1}}{2}\int\limits_{C_{n\delta}}\tr\left(W'_{n}|z|^{2W'_{n}}h\right)\displaystyle\left(\frac{dz}{z} - \frac{d\bar{z}}{\bar{z}}\right)\\
\intertext{and}
W_{\delta}[h] & = - \frac{\sqrt{-1}}{2}\iint\limits_{X_{\delta}}\tr \left(h^{*}(\theta_{1})\wedge h^{*}(\theta_{1}^{*})\right),
\end{align*}
and put
\[
S_{\delta}[h] = S_{0\delta}[h] + W_{\delta}[h] + 2\pi\log\delta \sum_{j=1}^{r}\sum_{i=1}^{n-1}\alpha^{2}_{ij} +2\pi\log\delta \sum_{j=1}^{r}\alpha_{nj}^{\prime 2}.
\]
We have the following result.
\begin{proposition}\label{WZNW function}
For any $\{E_{*}^{\rho}\}\in\cN_{0}$ and any $h\in \cC(E_{*}; X_0,\CMcal{H}_{r})$, the limit
\begin{equation}\label{eq:functional-limit}
S[h] = \lim_{\delta\rightarrow 0} S_{\delta}[h]
\end{equation}
is well-defined. The Euler-Lagrange equation for the functional 
$$S:\cC(E_{*}; X_{0},\CMcal{H}_{r})\rightarrow \R$$
is the zero-curvature equation \eqref{0-curv}. 
\end{proposition}
\begin{proof}
We will only prove that the limit \label{eq:functional-limit} exists. The derivation of the Euler--Lagrange equation is standard and accounts for the introduction of the boundary terms in $S_{0\delta}[h]$. The details are left to the reader. 

From formulas \eqref{as-1}-\eqref{as-2} we get the following asymptotic formulas in the neighborhood of each $z_{i}$:
\[
\tr(h^{-1}h_{z} h^{-1}h_{\bar{z}}) =\left\{
\begin{array}{ll}
\displaystyle\frac{1}{|z-z_{i}|^{2}}\left(\tr\left(W^{2}_{i}\right)+O\left(|z-z_{i}|^{\kappa_{i}}\right)\right) & i=1,\dots,n-1\\\\
\displaystyle\frac{1}{|z|^{2}}\left(\tr\left(W'^{2}_{n}\right)+O\left(\left|z\right|^{\kappa_{n}}\right)\right) & i=n
\end{array}
\right.
\]
where 
\[ 
 \kappa_{i}=\left\{
 \begin{array}{ll}
 2(\alpha_{i1}-\alpha_{ir}+1)>0 & i=1,\dots,n-1\\\\
 2\max\{\alpha'_{nj}-\alpha'_{nk}-2\}<0 & i=n
 \end{array}
 \right.
 \]
which determine the regularization of the kinetic term in the WZNW action. It is straightforward to verify that the limits as $\delta\rightarrow 0$ of the boundary terms of $S_{0\delta}[h]$ exist.
On the other hand, the topological term needs no regularization, i.e. the integral 
$$ \frac{\sqrt{-1}}{2}\iint\limits_{\CC\PP^{1}}\tr\left(b_{\bar{z}}b^{-1}(b_{\bar{z}}b^{-1})^{*}-b_{z}b^{-1}(b_{z}b^{-1})^{*}\right)dz\wedge d\bar{z}$$
is already absolutely convergent. Indeed, using the Cholesky decomposition  $h=c^{*}ac$ we obtain the following Cartan decompositions
\begin{equation*}
b h^{-1}h_{z} b^{-1}=u+d+l\quad\text{and} \quad bh^{-1}h_{\bar{z}}b^{-1}=l^{*}+d^{*}+u^{*},
\end{equation*}
where $$d=a_{z}a^{-1},\qquad u=a^{1/2}c_{z}c^{-1}a^{-1/2},\qquad l=(a^{1/2}c_{\bar{z}}c^{-1}a^{-1/2})^{*}$$ are the corresponding diagonal, upper triangular and lower triangular components. 
hence
\[
\tr(h^{-1}h_{z}h^{-1}h_{\bar{z}})
=\|d+u+l\|^{2}=\|d\|^{2}+\|u\|^{2}+\| l\|^{2},\qquad \text{where}\quad\|g\|^{2}=\tr(gg^{*}).
\] 
From \eqref{as-a} we obtain the following asymptotics 
\[
d(z,\bar{z})=\left\{
\begin{array}{ll}
\displaystyle\frac{W_{i}}{z-z_{i}}+O(1) & \quad\text{as}\quad z\rightarrow z_{i},\quad i=1,\dots, n-1,\\\\
-\displaystyle\frac{W'_{n}}{z}+O\left(\left|z^{-2}\right|\right) & \quad\text{as}\quad z\rightarrow\infty.
\end{array}
\right.
\]
Therefore the integrals 
$$\frac{\sqrt{-1}}{2}\iint\limits_{X_{\delta}}\|d\|^{2}dz\wedge d\bar{z} \qquad \text{and}\qquad \frac{\sqrt{-1}}{2}\iint\limits_{X_{\delta}}\tr(h^{-1}h_{z}h^{-1}h_{\bar{z}})dz\wedge d\bar{z}$$
diverge identically, so that 
$$\frac{\sqrt{-1}}{2}\iint\limits_{\CC\PP^{1}}(\|u\|^{2}+\|l\|^{2})dz\wedge d\bar{z}=\lim_{\delta\rightarrow 0}\frac{\sqrt{-1}}{2}\iint\limits_{X_{\delta}}(\|u\|^{2}+\|l\|^{2})dz\wedge d\bar{z}<\infty.$$
Since $\tr\left(b_{\bar{z}}b^{-1}(b_{\bar{z}}b^{-1})^{*}-b_{z}b^{-1}(b_{z}b^{-1})^{*}\right)=\|l\|^{2}-\|u\|^{2},$ we conclude that the topological term in the WZNW action is absolutely convergent.
\end{proof}

\begin{remark}
The density of the WZNW functional described before corresponds to an explicit realization of the second Bott--Chern form on a trivial Hermitian vector bundle \cite{PT14,Tak16}. It was already noted in \cite{Don92} that, in the case of a trivial vector bundle, the WZNW functional provides a concrete realization of Donaldson's functional \cite{Don85,Don87}. As a consequence of the Birkhoff--Grothendieck theorem, such a realization becomes meaningful in the study of moduli spaces of parabolic bundles in genus 0.
\end{remark}

\section{Main results} \label{Main}
Let $\cS:\cN_{0}\rightarrow \RR$ be the function defined on each $\{E^{\rho}_{*}\}\in\cN_{0}$ as the critical value of the functional $S:\cC(E_{*}; X_{0},\CMcal{H}_{r})\rightarrow \R$ in Proposition
\ref{WZNW function}. 
\subsection{The first variation of the WZNW action}
The following result is an analog of Theorem 1 in \cite{ZT87a} for vector bundles.
\begin{theorem}\label{main-theo1}
The function $\cS:\cN_{0}\to\mathbb{R}$ is smooth and satisfies
\begin{equation}\label{gen-funct}
\partial \cS=-\vartheta,
\end{equation}
where $\vartheta$ is the $(1,0)$-form on $\cN_{0}$ defined in \eqref{Theta-form}.
In other words, for every  $\nu\in \overline{\mathfrak{S}_{2}(\Ga,\Ad\rho)}$, 
\[
\left.\frac{\partial \cS(h^{\varepsilon\nu})}{\partial\varepsilon}\right|_{\varepsilon=0}=-2\iint\limits_{\Gamma\bk\HH}\tr\left(\mathscr{A}\nu\right)\text{d}^{2}\tau.
\]
\end{theorem}
\begin{proof} 
It is sufficient to verify that 
$$\left.\lim_{\delta\rightarrow 0}\frac{\partial S_{\delta}(h^{\varepsilon\nu})}{\partial\varepsilon}\right|_{\varepsilon=0}=-\vartheta(\nu),$$
uniformly in a neighborhood of  every $\{E_{*}\}\in\cN_{0}$, which we show by a straightforward, though tedious, computation. Namely, using Lemma \ref{var-F} and Stokes theorem, we have
\begin{align*}
\left.\frac{\partial S_{0\delta}(h^{\varepsilon\nu})}{\partial \varepsilon}\right|_{\varepsilon=0}
 =  - & \frac{\sqrt{-1}}{2}\iint\limits_{X_{\delta}}\tr\left(\dot{F}\left(2\bar{\partial}\left(h^{-1}\partial h\right) + h^{-1}dh\wedge h^{-1}dh \right)\right)\\
 + &\frac{\sqrt{-1}}{2}\int\limits_{\partial X_{\delta}}\!\tr\left(\!\dot{F}\left(2 h^{-1}\partial h - h^{-1} d h\right)\right)\\
 -  &\frac{\sqrt{-1}}{2}\sum_{i = 1}^{n - 1} \int\limits_{C_{i\delta}}\tr\left(W_{i}|z-z_{i}|^{-2W_{i}}\dot{\phi_{i}}_{+}\right)\displaystyle\left(\frac{dz}{z - z_{i}} - \frac{d\bar{z}}{\bar{z} - \bar{z}_{i}}\right)\\
 + & \frac{\sqrt{-1}}{2}\int\limits_{C_{n\delta}}\tr\left(W'_{n}|z|^{2W'_{n}}\dot{h}_{+}\right)\displaystyle\left(\frac{dz}{z} - \frac{d\bar{z}}{\bar{z}}\right), 
\end{align*}
where 
\[
\dot{\phi_{i}}_{+} = \left.\frac{\partial \left(C_{i}(0)^{\varepsilon\nu}\right)^{*} h^{\varepsilon\nu} C_{i}(0)^{\varepsilon\nu}}{\partial \varepsilon} \right|_{\varepsilon =0},\quad i = 1,\dots, n-1.
\]
Next, put
$$\dot{b}_{+}=\left.\frac{\del b^{\varepsilon\nu}}{\del \varepsilon}\right|_{\varepsilon=0}\quad\text{and}\quad \dot{b}_{-}=\left.\frac{\del b^{\varepsilon\nu}}{\del \bar\varepsilon}\right|_{\varepsilon=0}.$$
Using the identity $\dot{F}=-b^{-1}\left(\dot{b}_{+}b^{-1}+\left(\dot{b}_{-}b^{-1}\right)^{*}\right)b$, we obtain
\begin{align*}
 \left.\frac{\partial W_{\delta}(h^{\varepsilon\nu})}{\partial\varepsilon}\right|_{\varepsilon=0} = &
 \frac{\sqrt{-1}}{2}\iint\limits_{X_{\delta}}\tr\left(\dot{F} h^{-1} dh\wedge h^{-1} dh\right)\\
+ & \frac{\sqrt{-1}}{2}\int\limits_{\partial X_{\delta}}\tr\left(\left(\dot{b}_{-}b^{-1}\right)^{*}\theta_{1}-\dot{b}_{+}b^{-1}\theta_{1}^{*}\right)
\end{align*}

Using that $\bar{\partial}(h^{-1}\partial h) \equiv 0$, the identity $h^{-1}dh=b^{-1}(\theta_{1}+\theta_{1}^{\ast})b$, the relation $b=\sqrt{a}c$, and putting everything together, 
we obtain
\begin{align*}
\left.\frac{\partial S_{\delta}(h^{\vep\nu})}{\partial\varepsilon}\right|_{\varepsilon=0} = 
& \quad\frac{\sqrt{-1}}{2}\int\limits_{\partial X_{\delta}}\text{tr}\left(\dot{F}(2h^{-1}\partial h - h^{-1}dh)\right)\\
& +\frac{\sqrt{-1}}{2}\int\limits_{\partial X_{\delta}}\text{tr}\left(\left(\dot{b}_{-}b^{-1}\right)^{*}\theta_{1}-\dot{b}_{+}b^{-1}\theta_{1}^{*}\right)\\
& - \frac{\sqrt{-1}}{2}\sum_{i = 1}^{n - 1} \int\limits_{C_{i\delta}}\tr\left(W_{i}|z-z_{i}|^{-2W_{i}}\dot{\phi_{i}}_{+}\right)\displaystyle\left(\frac{dz}{z - z_{i}} - \frac{d\bar{z}}{\bar{z} - \bar{z}_{i}}\right)\\
& +  \frac{\sqrt{-1}}{2}\int\limits_{C_{n\delta}}\tr\left(W'_{n}|z|^{2W'_{n}}\dot{\phi_{n}}_{+}\right)\displaystyle\left(\frac{dz}{z} - \frac{d\bar{z}}{\bar{z}}\right)\\
 = & \quad \left\{\sqrt{-1}\int\limits_{\partial X_{\delta}}\text{tr}\left(\dot{F} h^{-1}\partial h\right) \right.\\
& - \frac{\sqrt{-1}}{2}\sum_{i = 1}^{n - 1} \int\limits_{C_{i\delta}}\tr\left(W_{i}|z-z_{i}|^{-2W_{i}}\dot{\phi_{i}}_{+}\right)\displaystyle\left(\frac{dz}{z - z_{i}} - \frac{d\bar{z}}{\bar{z} - \bar{z}_{i}}\right)\\
& \left. +  \frac{\sqrt{-1}}{2}\int\limits_{C_{n\delta}}\tr\left(W'_{n}|z|^{2W'_{n}}\dot{\phi_{n}}_{+}\right)\displaystyle\left(\frac{dz}{z} - \frac{d\bar{z}}{\bar{z}}\right)\right\}\\
& + \frac{\sqrt{-1}}{2}\int\limits_{\partial X_{\delta}}\tr\left(a^{-1}\dot{a}_{+}a^{-1}da\right) + \sqrt{-1}\int\limits_{\partial X_{\delta}}\tr\left(\left(\dot{c}_{-}c^{-1}\right)^{*}a dc c^{-1}a^{-1}\right)\\
= & \quad I_{1}+I_{2}+I_{3},
\end{align*}
where 
\[
\dot{a}_{+}=\left.\frac{\del a^{\varepsilon\nu}}{\del \varepsilon}\right|_{\varepsilon=0}\qquad \dot{c}_{-}=\left.\frac{\del c^{\varepsilon\nu}}{\del \bar\varepsilon}\right|_{\varepsilon=0}.
\]
Here we denoted by $I_{1}$ the sum of integrals inside the curly brackets, and by $I_{2}$ and $I_{3}$ the two integrals in the last line.
We will consider each of these integrals separately. Stokes theorem, a change of variables and \eqref{relation-F-nu}  give 
\[
\sqrt{-1}\int\limits_{\partial X_{\delta}}\text{tr}\left(\dot{F} h^{-1}\partial h\right) = \sqrt{-1}\iint\limits_{X_{\delta}}\text{tr}\left(\bar{\partial}\dot{F}\wedge h^{-1}\partial h\right) = -2\iint\limits_{D_{\delta}}\text{tr}\left(\mathscr{A}\nu\right)d^{2}\tau
\]
where $D_{\delta}=J^{-1}(X_{\delta})\cap D$. Moreover, since 
\[
\phi^{i}_{+} = \left\{
\begin{array}{cr}
o(1), & i = 1,\dots, n - 1,\\\\
o\left(|z|^{-1}\right), & i = n,
\end{array}
\right.
\]
it follows from the cusp form condition that $$I_{1}= -\vartheta(\nu) + o(1)\qquad \text{as} \qquad\delta\to 0.$$
In order to estimate the integral $I_{2}$, we observe that
\[
\tr\left(\dot{a}_{+} a^{-1} da a^{-1}\right) = \sum_{j=1}^{r}\dot{\left(\log a_{j}\right)}_{+}d \log a_{j},
\]
and it follows from \eqref{as-a} that each $\dot{\left(\log a_{j}\right)}_{+}$ is regular on $\CC\PP^{1}$. Hence the residue contributions from each summand of $\tr\left(\dot{a}_{+} a^{-1} da a^{-1}\right)$ cancel out individually. Hence $I_{2}=o(1)$ as $\delta\to 0$.

Next, we estimate the integral $I_{3}$. 
It follows from \eqref{as-a}--\eqref{as-c} that near $z_{1},\dots,z_{n-1}$,
\[
\tr\left((\dot{c}_{-}c^{-1})^{*}adcc^{-1}a^{-1}\right)
=\displaystyle\sum_{j<k}\left(\phi_{ijk}dz+\psi_{ijk}d\bar{z}\right)|z-z_{i}|^{2(\alpha_{ij}-\alpha_{ik})},
\]
where $\phi_{ijk},\psi_{ijk}\in\cL_{z_{i}}$, while near $\infty$,
\[
\tr\left((\dot{c}_{-}c^{-1})^{*}adcc^{-1}a^{-1}\right)
=\displaystyle\sum_{j<k}\left(\phi_{njk}dz+\psi_{njk}d\bar{z}\right)|z|^{-2(\alpha'_{nj}-\alpha'_{nk})},
\]
where $\phi_{njk},\psi_{njk}\in\cL_{z_{n}}$. It is readily verified that if $\phi\in\cL_{z_{i}}$, $i=1,\dots,n$, then 
\[
\int_{\partial\cU_{i}}\phi dz=|z-z_{i}|^2f(|z-z_{i}|),\qquad \int_{\partial\cU_{i}}\phi d\bar{z}=|z-z_{i}|^2g(|z-z_{i}|),
\]
with regular functions $f,g$. Consequently,  since for $j<k$, $\alpha_{ij}-\alpha_{ik}>-1$, and $\alpha'_{nj}-\alpha'_{nk}>-1$ as $\{E^{\rho}_{*}\}\in\cN_{0}$,  we conclude that $I_{3}=o(1)$ as $\delta\to 0$. Therefore,
\[
\left.\frac{\partial S_{\delta}}{\partial\varepsilon}\right|_{\varepsilon=0} = -\vartheta(\nu)+o(1)\quad\text{as}\quad \delta\to 0.
\]

Finally, notice that the function $Y_{\varepsilon\nu}$ is continuously differentiable in $\varepsilon$ and its Fourier coefficients have the same property (in particular this also holds for the constant terms). Since every single factor in the integrals $I_{1}$, $I_{2}$, and $I_{3}$ explicitly depends on the derivatives of $Y_{\varepsilon\nu}$ and $F^{\varepsilon\nu}$ at $\varepsilon=0$, we conclude that each of the corresponding remainders can be estimated uniformly in a neighborhood of $\{E_{*}\}$ in $\cN_{0}$.
\end{proof}

\subsection{The second variation of the WZNW action}
The next result is an analog of Theorem 1 in \cite{ZT89} for parabolic bundles.
\begin{theorem}\label{main-theo2}
The $(1,0)$-form $\vartheta$ on $\cN_{0}$ satisfies 
\begin{equation}\label{d-bar-section}
\bar{\partial}\vartheta=2\sqrt{-1}\left.\left(\Omega -\Omega_{\mathrm{T}}\right)\right|_{\cN_{0}},
\end{equation}
where 
$$\Omega_{\mathrm{T}}=\sum_{i = 1}^{n}\sum_{j=1}^{r}\beta_{ij}\Omega_{ij}$$
and
\begin{equation}\label{eq:beta1}
\beta_{ij}  = 2\pi^{2}\alpha_{ij}, \qquad i = 1,\dots, n-1,
\end{equation}
\begin{equation}\label{eq:beta2}
\beta_{nj}  = 
\begin{cases}
2\pi^{2}\left(\alpha_{nj}+ m_{j}\right) & \quad\text{if}\quad r\mid d\\\\
2\pi^{2}\left(\alpha_{nj}+ m_{r-j+1}\right) & \quad\text{if}\quad r\centernot\mid d 
\end{cases}
\end{equation}
\end{theorem}

\begin{proof} 
As in the proof of Theorem 1 in \cite{ZT89}, consider the family
\[
\mu^{\varepsilon\nu} = \mathrm{P}_{\varepsilon\nu}\left( \Ad f^{\varepsilon\nu} (\mu)\right),
\]
which for sufficiently small $\vep$ determines the tangent vector field $\partial/\partial\varepsilon(\mu)$ to the complex curve $\left\{E^{\rho^{\varepsilon\nu}}\right\}\subset \cN$. Let $\cA_{\vep\nu}$ be the corresponding family of weight $2$ regular forms associated with the $1$-form $\vartheta$, where $\cA_{0}=\cA$.
We have
\begin{align*}
\vartheta\left(\frac{\partial}{\partial\varepsilon(\mu)}\right) &=
\lim_{\delta\to 0}2\iint\limits_{D_{\delta}}\text{tr}\left(\mathscr{A}_{\varepsilon\nu}\mu^{\vep\nu}\right)d^{2}\tau \\
& =\lim_{\delta\to 0}2\iint\limits_{D_{\delta}}\text{tr}\left(\Ad (f^{\vep\nu})^{-1}(\mathscr{A}_{\varepsilon\nu})
\Ad (f^{\vep\nu})^{-1}(\mu^{\vep\nu})\right)d^{2}\tau
\end{align*}
and
\begin{equation}\label{derivative}
\left.\frac{\partial }{\partial\bar{\varepsilon}}\right|_{\varepsilon=0} \vartheta\left(\frac{\partial}{\partial\varepsilon(\mu)}\right)=\lim_{\delta\to 0}2\iint\limits_{D_{\delta}}\text{tr}\left(L_{\bar{\nu}}\cA\mu +\cA L_{\bar{\nu}}\mu \right)d^{2}\tau,
\end{equation}
where
$$L_{\bar{\nu}}\cA=\left.\frac{\partial }{\partial\bar{\varepsilon}}\right|_{\varepsilon=0} \Ad\left(f^{\varepsilon\nu}\right)^{-1}\left(A_{\vep\nu}\right).$$
It follows from \eqref{com-diagram} that 
\[
\left(f^{\varepsilon\nu}\right)^{-1}\cA_{\varepsilon\nu} f^{\varepsilon\nu}
+Y^{-1}\left(F^{\varepsilon\nu}\right)^{-1}
\left(F^{\varepsilon\nu}_{\tau}\circ J\right)Y
=\left(f^{\varepsilon\nu}\right)^{-1}f^{\varepsilon\nu}_{\tau}+\cA
\]
which, in virtue of variational formulas \eqref{var-formulas-2} and Lemma \ref{var-F}, implies that
\begin{equation}
\label{var-A-}
L_{\bar{\nu}}\cA 
=-\nu^{*}.
\end{equation}
Using formulas \eqref{eq:Lie-derivative} and \eqref{var-A-}, we obtain
\begin{align*}
\left.\frac{\partial }{\partial\bar{\varepsilon}}\right|_{\varepsilon=0} \vartheta\left(\frac{\partial}{\partial\varepsilon(\mu)}\right)
& = \lim_{\delta\to 0}\left(-2\iint\limits_{D_{\delta}}\text{tr}\left(\mu\nu^{*}\right)d^{2}\tau + 2\iint\limits_{D_{\delta}}\text{tr}\left(\cA  \frac{\partial f_{\mu\bar{\nu}}}{\partial \bar{\tau}}\right)d^{2}\tau\right)\\
&=-\langle \mu,\nu\rangle - \lim_{\delta\to 0} \sqrt{-1}\int\limits_{\partial D_{\delta}}\text{tr}\left(\cA  f_{\mu\bar{\nu}}\right)d\tau.
\end{align*}
It follows from the asymptotic formula \eqref{F-Fourier} and formula  \eqref{eq:asym-f-mu-nu} that
\[
\int\limits_{\partial D_{\delta}}\text{tr}\left(\cA  f_{\mu\bar{\nu}}\right)d\tau = \sum_{i = 1}^{n}\tr\left(\Ad(U_{i})(B_{i}(0))F^{i}_{\mu\bar{\nu}}\right) + o(1).
 \]
From the explicit form \eqref{eq:regular-Fourier-1}--\eqref{eq:regular-Fourier-2} of the matrices $B_{i}(0)$ in Lemma \ref{regular-auto} and formula \eqref{eq:u-v} we get
\[
\Ad(U_{i})\left(B_{i}(0)\right) = 2\pi\sqrt{-1}\left(\frac{\tr(W_{i})}{r}I_{r} +\displaystyle\sum_{j=1}^{r}\alpha_{ij}v_{ij}\right),\quad  i=1,\dots,n-1,
\]
and
\[
\Ad(U_{n})\left(B_{n}(0)\right) = \left\{
\begin{array}{ll}
2\pi\sqrt{-1}\left(\displaystyle\frac{\tr(W_{n} + N_{0})}{r}I_{r} + \sum_{j=1}^{r}(\alpha_{nj}+m_{j})v_{ij}\right) &  \quad\text{if}\quad r\mid d,\\\\
2\pi\sqrt{-1}\left(\displaystyle M_{0} + \sum_{j=1}^{r}(\alpha_{nj}+m_{r-j+1})v_{nj}\right) &  \quad\text{if}\quad r\centernot\mid d,
\end{array}
\right.
\]
where
\[
M_{0} = \dfrac{\tr(W_{n}+N'_{0})}{r}I_{r} -\Ad(U_{n})(L-I_{r}).
\]
Hence, considering the trace properties of $F^{i}_{\mu\bar{\nu}}$ described in section \ref{taut}, we obtain
\[
\tr\left(\Ad(U_{i})(B_{i}(0))F^{i}_{\mu\bar{\nu}}\right) = \frac{\sqrt{-1}}{\pi}\sum_{j=1}^{r}\beta_{ij}\tr\left(F^{i}_{\mu\bar{\nu}} v_{ij}\right),
\]
where $\beta_{ij}$ are given by \eqref{eq:beta1}--\eqref{eq:beta2}.
Finally, using \eqref{eq:cusp-Chern-form} we conclude that
\begin{gather*}
\delb\vartheta\left(\frac{\partial}{\partial\varepsilon(\mu)},\frac{\partial}{\partial\overline{\varepsilon(\nu)}}\right)=\left.\frac{\partial }{\partial\bar{\varepsilon}}\right|_{\varepsilon=0} \vartheta\left(\frac{\partial}{\partial\varepsilon(\mu)}\right)\\
=  2\sqrt{-1}\left(\Omega\left(\frac{\partial}{\partial\varepsilon(\mu)},\frac{\partial}{\partial\overline{\varepsilon(\nu)}}\right)
-\Omega_{\mathrm{T}}\left(\frac{\partial}{\partial\varepsilon(\mu)},\frac{\partial}{\partial\overline{\varepsilon(\nu)}}\right)\right).\qedhere
\end{gather*}
\end{proof}

\subsection{Applications}

\noindent Combining Theorems \ref{main-theo1} and \ref{main-theo2}, we obtain 
\begin{corollary}\label{cor:potential}
The real-valued function $-\cS:\cN_{0}\rightarrow \RR$ is a K\"ahler potential for $\left.\left(\Omega -\Omega_{\mathrm{T}}\right)\right|_{\cN_{0}}$, the restriction of the difference between the natural K\"{a}hler form $\Omega$ and $\Omega_{\mathrm{T}}$ on $\cN$ to the regular locus $\cN_{0}$, i.e.,
\begin{equation}\label{eq:potential}
\partial\bar{\partial} \cS = 2\sqrt{-1}\left.\left(\Omega - \Omega_{\mathrm{T}}\right)\right|_{\cN_{0}}.
\end{equation}
In particular, on $\cN_{0}$ we have the following identity of de Rham cohomology classes 
\[
\left[\Omega\right] = \left[\Omega_{\mathrm{T}}\right].
\]
Furthermore, the result holds globally on any open chamber of parabolic weights for which $\cN_{0} = \cN$. 
\end{corollary}

\begin{remark}
The computation of the first Chern form for moduli spaces of stable parabolic bundles determined in \cite[Corollary 1]{TZ07} gives
\[
c_{1}\left(\lambda,\|\cdot\|_{Q}\right) = -\frac{r}{\pi^{2}}\Omega - \sum_{i = 1}^{n}\sum_{\substack{j,k=1\\ j\neq k}}^{r}\textrm{sgn}(\alpha_{ij}-\alpha_{ik}) (1- 2|\alpha_{ij}-\alpha_{ik}|)\Omega_{ij}
\]
(the formula applies in full generality, since the determinant of a vector bundle on $\CC\PP^{1}$ is always fixed in the moduli problem). In the simplest case $r=2$, we have the additional relation $\Omega_{i2} = -\Omega_{i1}$. In genus 0, the previous identity can be compared with \eqref{eq:potential}, yielding the following relation for the first Chern form of $\det T^{*}\cN$ on any open weight chamber for which $\cN_{0} = \cN$: 
\begin{equation}
c_{1}\left(\det T^{*}\cN\right) =
\begin{cases}
-\displaystyle 2\sum_{i = 1}^{n} c_{1}\left(\ell_{i2}\right) & \text{if \emph{d} is even},\\\\
-\displaystyle 2\sum_{i = 1}^{n - 1} c_{1}\left(\ell_{i2}\right) + 2c_{1}\left(\ell_{n2}\right) & \text{if \emph{d} is odd}.
\end{cases}
\end{equation}
\end{remark}

\begin{remark}
Corollary \ref{cor:potential} reduces the computation of symplectic volumes for the Narasimhan--Atiyah--Bott K\"ahler form to the intersection theory of tautological forms  whenever $\cN_{0} = \cN$ (cf. \cite{TZ07}). These volumes were first computed by Witten \cite{W91} for the group $\mathrm{SU}(2)$ in terms of the Verlinde formula (note that the piecewise-polynomial volume dependence on the parabolic weights is concealed in the explicit form of Witten's computation). It follows that a sensible construction of explicit geometric models for moduli spaces of parabolic bundles on the sphere, as well as the algebraic geometry of their tautological classes, would serve to reduce the problem of computation of symplectic volumes to a combinatorial one, while making the behavior of the latter under wall-crossing explicit. The implementation of this idea is under investigation by the first author, and will appear separately.
\end{remark}

\bibliographystyle{amsalpha}
\bibliography{bundles}
 \end{document}